\documentclass{amsart}

\title{Partial bases and homological stability of $\GL_{n}(R)$ revisited}

\author{Calista Bernard}
\address{School of Mathematics, University of Minnesota, USA}
\email{berna249@umn.edu}

\author{Jeremy Miller}
\address{Department of Mathematics, Purdue University, USA}
\email{jeremykmiller@purdue.edu}

\author{Robin J. Sroka} 
\address{Mathematisches Institut, Universität Münster, Germany}
\email{robinjsroka@uni-muenster.de}

\thanks{CB and RJS were supported by the Swedish Research Council under grant no.\ 2016-06596 while in residence at Institut Mittag-Leffler in Djursholm, Sweden during the semester \emph{Higher algebraic structures in algebra, topology and geometry}. JM was supported by NSF grant DMS-2202943. RJS was supported by NSERC Discovery Grant A4000 in connection with a Postdoctoral Fellowship at McMaster University, by the Deutsche Forschungsgemeinschaft (DFG, German Research Foundation) -- Project-ID 427320536 -- SFB 1442, as well as by Germany’s Excellence Strategy EXC 2044 -- 390685587, Mathematics Münster: Dynamics–Geometry–Structure.}

\input{latex-config/packages.tex}

\newcommand{\m}{\to}

\providecommand{\Span}{\ensuremath\mathsf{span}}

\providecommand{\Star}{\ensuremath\mathsf{Star}}
\providecommand{\Link}{\ensuremath\mathsf{Link}}

\providecommand{\N}{\ensuremath\mathbb N}

\providecommand{\Z}{\ensuremath\mathbb Z}

\providecommand{\Q}{\ensuremath\mathbb Q}

\providecommand{\F}{\ensuremath\mathbb F}

\providecommand{\bT}{\ensuremath\mathbb T}

\DeclareMathOperator{\im}{im}

\DeclareMathOperator{\Hom}{Hom}

\DeclareMathOperator{\End}{End}

\DeclareMathOperator{\GL}{GL}

\DeclareMathOperator{\SL}{SL}
\DeclareMathOperator{\rank}{rank}

\DeclareMathOperator{\ord}{ord}

\DeclareMathOperator{\Ind}{Ind}

\DeclareMathOperator{\St}{St}

\newcommand{\U}{\operatorname{U}}
\newcommand{\bU}{\operatorname{\mathbb{U}}}
\newcommand{\B}{\operatorname{B}}
\newcommand{\bB}{\operatorname{\mathbb{B}}}
\newcommand{\BA}{\operatorname{BA}}

\newcommand{\BX}{\operatorname{BX}}

\newcommand{\bLink}{\operatorname{\mathbb{Link}}}
\newcommand{\bS}{\operatorname{\mathbb{S}}}

\newcommand{\relstgen}{\mathcal{S}}
\newcommand{\fix}{\operatorname{fix}}
\newcommand{\K}{\mathbb{k}}
\newcommand{\pres}{\operatorname{pres}}

\providecommand{\incl}{\ensuremath\mathsf{incl}}

\newcommand{\on}[1]{\operatorname{#1}}

\begin{document}
	
	\begin{abstract}
		Let $R$ be a unital ring satisfying the invariant basis number property, that every stably free $R$-module is free, and that the complex of partial bases of every finite rank free module is Cohen--Macaulay. This class of rings includes every ring of stable rank $1$ (e.g.\ any local, semi-local or Artinian ring), every Euclidean domain, and every Dedekind domain $\mathcal{O}_S$ of arithmetic type where $|S| > 1$ and $S$ contains at least one non-complex place. Extending recent work of Galatius--Kupers--Randal-Williams and Kupers--Miller--Patzt, we prove that the sequence of general linear groups $\GL_n(R)$ satisfies slope-$1$ homological stability with $\Z[1/2]$-coefficients.
	\end{abstract}

	\maketitle
	
	\setcounter{tocdepth}{1}
	\tableofcontents
	
	\emergencystretch=2em
	\section{Introduction}
	\label{section-1}

Let $R$ be a unital ring. Charney \cite{charney1980}, Maazen \cite{maazen79} and van der Kallen \cite{vanderkallen1980} proved generic slope-$1/2$ homological stability results that apply to large classes of rings $R$, including all Euclidean and Dedekind domains. The most general version, due to van der Kallen \cite{vanderkallen1980}, shows that if the Bass stable rank of $R$ is finite, $\on{sr}(R) < \infty$, then the inclusion induced map
\[
H_i(\GL_{n-1}(R); \Z) \xrightarrow{\cong} H_i(\GL_{n}(R); \Z)
\] 
is an isomorphism in the slope-$1/2$ range given by $i \leq n/2+c$, where $c \in \Q$ is a constant depending on $\on{sr}(R)$. To establish this, and in the process resolving a conjecture of Quillen (cf.\ \cite[Section 1]{wagoner1976}), van der Kallen proved that certain complexes of unimodular vectors in $R^n$ are highly connected for every $n \geq 0$.

The goal of this work is to show that in many cases of interest, e.g.\ if $R$ is a local, semi-local or Artinian ring; a Euclidean domain; or a Dedekind domain $R = \mathcal{O}_S$ of arithmetic type where $|S| > 1$ and $S$ contains at least one non-complex place, the slope-$1/2$ ranges established in \cite{charney1980, maazen79, vanderkallen1980} can be improved to slope-$1$ ranges if one is willing to invert $2$ in the coefficient module.

\subsection{Statement of main result} 

Our argument requires that the ring $R$ satisfy three properties. To state these, we recall the definition of a well-studied simplicial complex \cite{quillen1974, wagoner1976, maazen79, vanderkallen1980, churchfarbputman2019}, which is closely related to the aforementioned complexes of unimodular vectors (see \cref{sec:assumptions}).

\begin{definition}
	\label{definition:partial-bases-poset}
	A \emph{partial basis} of $R^n$ is a subset of a basis. The $n$-th \emph{complex of partial bases} of $R$, denoted by $B_n(R)$, is the simplicial complex in which a $k$-simplex is a partial basis of $R^n$ of size $k+1$.
\end{definition}

Recall that a simplicial complex $X$ is called \emph{$d$-spherical} if $X$ is of finite dimension $\dim(X) = d < \infty$ and $(d-1)$-connected. Following Quillen \cite[Section 8]{quillen1978}, we say that $X$ is \emph{Cohen--Macaulay} if it is of finite dimension $\dim(X) = d < \infty$, $d$-spherical and the link of every $p$-simplex in $X$ is $(d - p - 1)$-spherical. Using this topological notion and well-known ring-theoretic properties, we can now state our assumptions.

\begin{assumption}
	\label{assumption:R}
	Let $R$ be a unital ring such that
	\begin{enumerate}
		\item \label{assumption-invariant-basis} $R$ has the invariant basis number property;
		\item \label{assumption-stably-free} $R$ is Hermite (i.e.\ every stably free $R$-module is free);
		\item \label{assumption-partial-bases} for all $n \geq 0$ the complex of partial bases $B_n(R)$ is Cohen--Macaulay.
	\end{enumerate}
\end{assumption}
	
Our main result then takes the following form.

\begin{theorem}
	\label{theoremA} 
	Let $R$ be a unital ring satisfying \cref{assumption:R} and let $\K$ be a commutative ring in which $2 \in \K^\times$ is a unit. Then
	\[
	H_i(\GL_{n-1}(R);\K) \m H_i(\GL_{n}(R);\K)
	\]
	is an isomorphism for $i \leq n-2$ and a surjection for $i\leq n-1$.  
\end{theorem}

Previously, such slope-$1$ homological stability results were known to hold for general linear groups over fields and connected semi-local rings with infinite residue fields with $\K = \Z$ \cite{quillen1974, suslin1984, nesterenkosuslin1989, guin1989, galatiuskupersrandalwilliams2018cellsandfinite, galatiuskupersrandalwilliams2020cellsandinfinite, sprehnwahl2020}; the integers, the Gaussian integers, and the Eisenstein integers with $\K = \Z[1/2]$ \cite{kupersmillerpatzt2022}; and any ring of integers of a number field with $\K = \Q$ \cite{lisun2019}.

Our theorem recovers almost all of these results with $\K = \Z[1/2]$; the only exception is the case of totally imaginary, non-Euclidean number rings \cite{lisun2019}. Moreover, it extends the class of rings whose general linear groups satisfy slope-$1$ homological stability with $\Z[1/2]$-coefficients to include:

\begin{enumerate}[leftmargin=*]
	\item \emph{Every ring of stable rank 1} (e.g.\ all fields; all local, semi-local or Artinian rings \cites[(6.5)]{bass1964}[Example 1.1]{vaserstein1984}; the ring of all algebraic integers \cite[Example 1.2]{vaserstein1984}):\\
	\cref{assumption-invariant-basis} is an elementary consequence of having finite stable rank, see e.g.\ \cite[Remark 5.8]{randalwilliamswahl2017}. \cref{assumption-stably-free} follows from the fact that the stable rank is equal to one, see \cite[Corollary 11.1.5]{fajardoetal2020}. \cref{assumption-partial-bases} holds because \cref{assumption-invariant-basis} and \cref{assumption-stably-free} imply (compare \cref{sec:assumptions}) that $B_n(R)$ agrees with the complex of unimodular sequences investigated by van der Kallen in \cite{vanderkallen1980}, who proved that these complexes are Cohen--Macaulay \cite[§2. An Acyclicity Theorem]{vanderkallen1980}.
	\item \emph{Every Euclidean domain} (e.g.\ Dedekind domains with finitely many prime ideals, Euclidean number rings,  the ring of polynomials over any field):\\
	\cref{assumption-invariant-basis} holds for every commutative ring, see e.g.\ \cite[Theorem 2.6]{cohn1966}. \cref{assumption-stably-free} follows from the fact that Euclidean domains are principal ideal domains. \cref{assumption-partial-bases} is a result of Maazen \cite{maazen79}, which inspired van der Kallen's work \cite{vanderkallen1980}.
	\item \emph{Every Dedekind domain $R = \mathcal{O}_S$ of arithmetic type, where $|S| > 1$ and $S$ contains at least one non-complex place} (e.g.\ number rings $R \neq \Z$ with a real embedding or rings of $S$-integers of global fields with infinitely many units in which some prime is invertible \cite[p.1380]{churchfarbputman2019}):\\
	\cref{assumption-invariant-basis} holds for every commutative ring, see e.g.\ \cite[Theorem 2.6]{cohn1966}. \cref{assumption-stably-free} follows from the Steinitz-Chevalley structure theory for Dedekind domains, see e.g.\ \cite[Examples 4.7 (4)]{lam2006}. \cref{assumption-partial-bases} is a result of Church--Farb--Putman \cite[Theorem E]{churchfarbputman2019}. 
\end{enumerate}

In the last item and following \cite[Definition 1.3]{churchfarbputman2019}, by a Dedekind domain of arithmetic type $R$ we mean the ring of $S$-integers $\mathcal{O}_S \coloneqq \{x \in K : \ord_{\mathfrak{p}}(x) \geq 0 \text{ for all } \mathfrak{p} \in S\}$ of a global field $K$ (i.e.\ a number field $K/\Q$ or a function field in one variable over a finite field $K/\mathbb{F}_q(T)$), where $S$ is a finite nonempty set of places of $K$ (containing all infinite places if $K$ is a number field). This includes all number rings and certain localizations thereof.
As explained in \cite{churchfarbputman2019}, the assumption that $|S|> 1$ is equivalent to $R$ having infinitely many units. In particular, the third class of rings above does not include Dedekind domains such as the Euclidean number rings $\Z$ and $\Z[i]$ (which are covered by the second class). More generally, the third class does not include rings of integers of totally imaginary number fields and those $\mathcal{O}_S$ with finitely many units in the function field case (e.g.\ $\mathbb{F}_p[T]$).

Our theorem improves the generic slope-$2/3$ stability result with $\Z[1/2]$-coefficients for general linear groups over Euclidean domains obtained by Kupers--Miller--Patzt \cite[Theorem B]{kupersmillerpatzt2022}. For Dedekind domains of class number $1$ and if one is willing to invert $2$ in the coefficients, it often improves a generic slope-$2/3$ stability result due to Galatius--Kupers--Randal-Williams \cite[Section 18.2]{galatiuskupersrandalwilliams2021cells}.

At this level of generality, the slope-$1$ range in \cref{theoremA} cannot be improved: Results of Suslin \cite{suslin1984}, Nesterenko--Suslin \cite{nesterenkosuslin1989} and Guin \cite{guin1989} show e.g.\ that for infinite fields $F$ it holds that $H_n(\GL_n(F),\GL_{n-1}(F); \Z)$ is isomorphic to the $n$-th Milnor $K$-theory group $K^M_n(F)$. Since these are often nonzero for all $n$, slope-$1$ stability is often sharp.

We furthermore remark that it is necessary to use coefficients where $2$ is invertible in \autoref{theoremA}, because e.g.\ $H_1(\GL_2(\F_2),\GL_1(\F_2); \Z) \neq 0$.

\subsection{Outline and main technical achievement}

The proof of \cref{theoremA} builds on ideas due to Charney \cite{charney1980}, Galatius--Kupers--Randal-Williams \cite{galatiuskupersrandalwilliams2018cellsandfinite} and Kupers--Miller--Patzt \cite{kupersmillerpatzt2022}. Indeed, our general strategy of proof builds on that used by Kupers--Miller--Patzt in \cite{kupersmillerpatzt2022} to prove slope-1 stability with $\Z[1/2]$-coefficients for the general linear groups over the integers, the Gaussian integers and the Eisenstein integers. \cite{kupersmillerpatzt2022} use ideas developed in \cite{galatiuskupersrandalwilliams2018cellsandfinite, charney1980} to reduce the question of whether slope-1 stability with $\Z[1/2]$-coefficients holds to the question of whether the co-invariants of certain modules associated to $\{\GL_{n}(R)\}_{n \in \N}$ are trivial (after tensoring with $\Z[1/2]$). The relevant modules are the Steinberg module, denoted by $\St_n(R)$, and certain \emph{relative} Steinberg modules, denoted by $\St_n^m(R)$.

Steinberg modules $\St_n(R)$ play an important role in representation theory (e.g.\ \cite{steinberg1951}), the study of the cohomology of arithmetic groups (e.g.\ \cite{borelserre1973}) and the algebraic $K$-theory of $R$ (e.g.\ \cite{quillen1973}). If \cref{assumption:R} holds, an argument due to Church--Farb--Putman \cite[Proof of Theorem A]{churchfarbputman2019} shows that the complex of partial bases $B_{n}(R)$ can be used to construct a generating set of the Steinberg module $\St_n(R)$ and that this generating set can be used to check that its co-invariants vanish, 
\[
	(\St_n(R) \otimes \Z[1/2])_{\GL_{n}(R)} = 0, \text{ if } n \geq 2.
\]
This was carried out by Scalamandre \cite{scalamandre2023} in the generality that we require in this work.

Our main technical result uses a new induction procedure to show that \cref{assumption:R} also suffices to construct generating sets for the \emph{relative} Steinberg modules $\St_n^m(R)$, which can be used to check that their co-invariants also have the desired vanishing property,
\[
(\St_n^m(R) \otimes \Z[1/2])_{\GL_n^m(R)} = 0, \text{ if } m \geq 1 \text{ and } n \geq 1.
\]
Previously, it was only known how to construct such a generating set for \emph{relative} Steinberg modules under significantly stronger connectivity assumptions: Kupers--Miller--Patzt's approach in \cite{kupersmillerpatzt2022} requires that certain complexes of augmented partial frames $\BA_n(R)$, introduced by Church--Putman \cite{churchputman2017}, are Cohen--Macaulay. Here $\BA_n(R)$ is a simplicial complex of dimension $n$, while the complexes of partial bases $B_{n}(R)$ appearing in \cref{assumption:R} are only $(n-1)$-dimensional. To satisfy the Cohen--Macaulay property, the complex of augmented frames $\BA_n(R)$ ought therefore to be $(n-1)$-connected, while \cref{assumption-partial-bases} in \cref{assumption:R} only requires $B_n(R)$ to be $(n-2)$-connected. To date, the complexes of augmented partial frames are only known to have the desired connectivity properties if $R$ is a field or one of three rings: the integers \cite{churchputman2017}, the Gaussian and Eisenstein integers \cite{kupersmillerpatztwilson2022}. In fact, it is known that for many Euclidean rings the complexes of augmented partial frames $\BA_n(R)$ do \emph{not} satisfy the connectivity assumption needed to run the argument of Kupers--Miller--Patzt, see \cite[Proof of Theorem C]{kupersmillerpatztwilson2022}. 

The main innovation of this work is hence to bypass the need for high connectivity of the complexes of augmented partial frames $\BA_n(R)$ in \cite{kupersmillerpatzt2022}: We use the \emph{same} simplicial complexes $B_{n}(R)$ as Maazen \cite{maazen79} and van der Kallen \cite{vanderkallen1980}, yet are still able to double the slope of the stable range.

Our homological stability argument is powered by the connection to $E_k$-cells introduced in \cite[Theorem 4.2]{kupersmiller2018}, and the relation between $E_k$-cells and splitting complexes developed by Galatius--Kupers--Randal-Williams in \cite{galatiuskupersrandalwilliams2021cells} and \cite{galatiuskupersrandalwilliams2018cellsandfinite}. As in \cite{galatiuskupersrandalwilliams2018cellsandfinite}, the two vanishing results for the co-invariants of Steinberg and relative Steinberg modules lead to a vanishing result for the co-invariants of the \emph{Charney modules} $\on{Ch}_n(R)$ after tensoring with $\Z[1/2]$ (see \cref{def:charney-module} and \cref{proposition:computational-input-split-steinberg}),
\[
	(\on{Ch}_n(R) \otimes \Z[1/2])_{\GL_n(R)} = 0 \text{ if } n \geq 2.
\]
The $\GL_n(R)$-modules $\on{Ch}_n(R)$ arise as the top-degree homology groups of spherical splitting complexes, which were first studied in Charney's work on homological stability of general linear groups \cite{charney1980}. More recently, Galatius--Kupers--Randal-Williams proved that the homology of the groups $\GL_n(R)$ with coefficients in the Charney modules measures the $E_1$-André--Quillen homology of the $E_\infty$-algebra $\on{BGL}(R) = \bigsqcup_{n \in \N} \on{BGL}_n(R)$ (compare \cref{remark:charney-module-is-e1-steinberg}). Proving a vanishing result for $H_0(\GL_n(R); \on{Ch}_n(R) \otimes \Z[1/2])$, i.e.\ the co-invariants above, is the required input for deducing the slope-1 stability result using the theory of $E_k$-cells (see \cite[Proposition 5.1]{kupersmillerpatzt2022}).

\subsection{Outline} In \cref{sec:assumptions}, we elaborate on the ring-theoretic meaning of \cref{assumption:R} and give an equivalent set of assumptions. We show that $R$ satisfies \cref{assumption:R} if and only if $R^{op}$ does. This left-right duality is a key technical ingredient in our stability argument. In \cref{sec:relative-steinberg-modules}, we introduce the Steinberg and relative Steinberg modules. In \cref{sec:generating-sets-for-relative-steinberg-modules}, we prove the desired vanishing results for the co-invariants of the relative Steinberg modules. This is the main technical innovation of this work. In \cref{sec:slope-1-homological-stability}, we deduce our main theorem. We show the vanishing result for the co-invariants of the Charney modules and deduce slope-1 homological stability using the theory of $E_k$-cells. In \cref{sec:appendix}, we collect all technical arguments that need to be carried out but that are similar to arguments contained in the literature. Its function is to keep the article focused on the novel contributions.

\subsection{Acknowledgements}

It is a pleasure to thank Jian-Shu Li, Binyong Sun and Oscar Randal-Williams for helpful correspondence about the relation of this article and \cite{lisun2019}, and Alexander Kupers for helpful conversations.

\subsection{Notation and conventions}

$R$ denotes a unital (not necessarily commutative) ring, and all $R$-modules are taken to be left $R$-modules. We denote by $\langle S \rangle$ the $R$-span of a subset $S$ of an $R$-module. If $M$ is a left $R$-module, we denote by $\End_R(M)$ its ring of left $R$-module endomorphisms, where the addition is defined point-wise and we use the convention that the product $(f \cdot g)$ is given by composition in the opposite order $(g \circ f)$. With this convention $M$ is an $(R, \End_R(M))$-bimodule, and if $M$ is a free $R$-module of rank $n$, there are ring isomorphisms $\End_R(_RR) \cong R$ and, picking an $R$-module basis, $\End_R(M) \cong \on{Mat}_{n \times n}(R)$. We will henceforth write morphisms and matrices acting on left modules on the right (and vice versa). We use blackboard bold to denote posets $\mathbb{P}$. If $x \in \mathbb{P}$, we write $\mathbb{P}_{< x} \coloneqq \{z \in \mathbb{P}: z < x\}$ and similarly define $\mathbb{P}_{\leq x}, \mathbb{P}_{> x}$ and $\mathbb{P}_{\geq x}$. If $x < y$ in $\mathbb{P}$, we write $\mathbb{P}_{(x,y)} \coloneqq \{z \in \mathbb{P}: x < z < y\}$.
	\section{Assumptions on $R$}
	\label{sec:assumptions}
In this section, we describe a set of assumptions on $R$ that is equivalent to \cref{assumption:R}, but expressed in terms of the spaces of unimodular sequences appearing in \cite{quillen1978,wagoner1976,maazen79,vanderkallen1980}. We use this to discuss our assumptions on $R$ in greater detail, and to elaborate their ring-theoretic meaning. Combined with an argument carried out in \cref{appendix:partial-bases-of-the-opposite-ring}, we prove that $R$ satisfies \cref{assumption:R} if and only if its opposite ring $R^{op}$ does (see \cref{theorem:relation-to-Rop}). This left-right duality is an important ingredient in the proof of the key \cref{proposition:computational-input-split-steinberg}; it allows to apply a dualizing trick due to Charney \cite{charney1980} on which the strategy of \cite{kupersmillerpatzt2022} relies. This is the only section in this work in which we do not assume that \cref{assumption:R} holds.

\begin{definition}
	\label{definition:complex-of-unimodular-sequences}
	Let $M$ be a free $R$-module. A vector $\vec v \in M$ is called \emph{unimodular} if it is the basis of a free direct summand of rank-1 in $M$. The \emph{complex of unimodular vectors} $\U(M)$ is the simplicial complex whose vertices are unimodular vectors in $M$, and where a collection of vectors forms
	a simplex if and only if they are a basis for a free direct summand of $M$. Let $\U_n(R)$ denote $\U(R^n)$. We write $\bU_n(R)$ for the simplex poset of $\U(R)$, and call it the poset of unimodular vectors.
\end{definition}

Note that the poset of unimodular vectors $\bU_n(R)$ in \cref{definition:complex-of-unimodular-sequences} a priori differs from the poset of partial bases $\bB_n(R)$, i.e.\ the simplex poset of the complex $B_n(R)$ in \cref{definition:partial-bases-poset}, since e.g.\ in general a set of $n$ unimodular (i.e.\ linearly independent) vectors in $R^n$ need not be a basis, and complements of free summands need not be free. 

However, we will show that \cref{assumption:R} is equivalent to the following set of assumptions, and that under these assumptions it holds that $\bB_n(R) = \bU_n(R)$ for all $n \geq 0$.

\begin{assumption}\label{assumption:R-2}
	$R$ is a unital ring such that
	\begin{enumerate}
		\item \label{assumption-weakly-finite} $R$ is weakly finite (i.e.\ if $n \geq 0$ and $R^n\cong R^n \oplus C$, then $C=0$);
		\item \label{assumption-unimodular-sequences} for all $n \geq 0$ the complexes of unimodular vectors $\U_n(R)$ is Cohen--Macaulay.
	\end{enumerate}
\end{assumption}

In the next two subsections, we establish and discuss the following result.

\begin{proposition}
	\label{proposition:assumptions-are-equivalent}
	Let $R$ be a unital ring. Then \cref{assumption:R} holds if and only if \cref{assumption:R-2} holds, and in either case $\bB_n(R) = \bU_n(R)$ for all $n \geq 0$.
\end{proposition}

\subsection{Invariant basis number property and other rank conditions} 

The following three properties are common assumptions on a unital ring $R$ in algebra and topology \cite{cohn1966}:

\begin{enumerate}[label=(\Roman*), ref=\Roman*]
	\item \label[property]{ibn-I} For all $m,n \in \N$ it holds that $R^m \cong R^n$ implies $m = n$.
	\item \label[property]{ibn-II} For all $m,n \in \N$ it holds that $R^m \cong R^n \oplus C$ implies $m \geq n$.
	\item \label[property]{ibn-III} For all $n \in \N$ it holds that $R^n \cong R^n \oplus C$ implies $C = 0$.
\end{enumerate}

\cref{ibn-I} is called the \emph{invariant basis number property} (i.e.\ this is \cref{assumption-invariant-basis} of \cref{assumption:R}), and \cref{ibn-III} is called \emph{weakly or stable finiteness} in the literature \cite[Chapter 0.1]{cohn2006} (i.e.\ this is \cref{assumption-weakly-finite} of \cref{assumption:R-2}). We remark that \cref{ibn-III} is equivalent to saying that a set of $n$ linearly independent vectors in $R^n$ (i.e.\ an $(n-1)$-simplex in the complex of unimodular vectors $\U_n(R)$) is a basis of $R^n$. It is an exercise to see that \cref{ibn-III} implies \cref{ibn-II}, and that \cref{ibn-II} implies \cref{ibn-I}. The reverse implications are false in general \cite{cohn1966}. Nevertheless, \cref{ibn-I}, \cref{ibn-II} and \cref{ibn-III} are equivalent if every stably free $R$-module is free \cite[Theorem 2.7]{cohn1966}. Our discussion yields the following.

\begin{corollary}
	\label{corollary:assumption-consequence-1}
	If $R$ satisfies \cref{assumption:R} or \cref{assumption:R-2}, then $R$ satisfies \cref{ibn-I}, \cref{ibn-II} and \cref{ibn-III}. In particular, $\B_n(R)$ and $\U_n(R)$ are $(n-1)$-dimensional.
\end{corollary}

\subsection{Hermite rings and completing unimodular sequences}

A ring is called \emph{Hermite} if every stably free module is free \cite{lam2006}. While the rank conditions discussed in the previous subsection are satisfied by most rings that one commonly encounters, the assumption that our ring $R$ is Hermite, i.e.\ \cref{assumption-stably-free} in \cref{assumption:R}, is more restrictive (see e.g.\ \cite[Chapter 0.4]{cohn2006}). It has the following equivalent characterization: If $R$ satisfies \cref{ibn-I}, \cref{ibn-II} or \cref{ibn-III}, then $R$ is Hermite if and only if $R^m \oplus C = R^n$ implies that $m \leq n$ and $C \cong R^{n-m}$ \cite[Corollary 0.4.2]{cohn2006}. In terms of unimodular vectors, this can be phrased as follows.

\begin{lemma}
	\label{lemma:hermite-partial-basis-and-unimodular-sequences}
	Let $R$ be a unital ring satisfying \cref{ibn-III}. Then $R$ is Hermite if and only if the link of every $k$-simplex in $\U_n(R)$ is nonempty for $k \leq n-2$. In this case, it holds that $\B_n(R) = \U_n(R)$ and that this complex is pure of dimension $(n-1)$ (i.e.\ every simplex is contained in a maximal $(n-1)$-simplex).
\end{lemma}

\begin{proof}
	Let $\Delta = \{\vec v_0,\dots,\vec v_k\}$ be a $k$-simplex in $\U_n(R)$ for $k \leq n-2$, let $V$ be the free rank-$(k+1)$ summand it spans in $R^n$ and let $C$ be a complement of $V$ in $R^n = V \oplus C$. 
	
	Assume that $R$ is Hermite. Then \cite[Corollary 0.4.2]{cohn2006} (stated above) implies $C$ is a free module of rank-$(n - k - 1)$. Picking any basis $\{\vec c_1, \dots, \vec c_{n-k-1}\}$ of $C$, it follows that $\{\vec v_0,\dots,\vec v_k\} \sqcup \{\vec c_1, \dots, \vec c_{n-k-1}\}$ is a basis of $R^n$. In particular, $\B_n(R) = \U_n(R)$, the link of $\Delta$ in $\U_n(R)$ is nonempty (since $n - k - 1 \geq 1$) and $\Delta$ is contained in an $(n-1)$-simplex.
	
	If the link of $\Delta$ in $\U_n(R)$ is nonempty, then there exists a $\vec c_{1}\in R^n$ such that $\Delta' = \{\vec v_0,\dots,\vec v_k\} \sqcup \{\vec c_{1}\}$ is a $(k+1)$-simplex in $\U_n(R)$. Repeating this argument $(n-k-1)$-many times, we obtain an maximal $(n-1)$-simplex $\{\vec v_1,\dots,\vec v_{k+1}\} \sqcup \{\vec c_1, \dots, \vec c_{n-k-1}\}$, which is a basis of $R^n$ by \cref{ibn-III}. Hence, $\B_n(R) = \U_n(R)$, $\Delta$ is contained in an $(n-1)$-simplex and $V = \langle \Delta \rangle$ has a free complement in $R^n$. To see that $R$ is Hermite, let $C'$ be any stably free $R$-module. If $C'$ is not finitely generated, then $C'$ has to be free \cite[Proposition 4.2]{lam2006}. If $C'$ is finitely generated, then $C'$ is the complement $V \oplus C' = R^n$ for some free summand $V = \langle \Delta \rangle$, spanned by some $\Delta$ and in some $R^n$. Since all complements of $V$ are isomorphic (to the quotient $R^n/V$) and $V$ has a free complement, $C'$ is free as well.
\end{proof}

This also completes the proof of \cref{proposition:assumptions-are-equivalent}, since \cref{assumption-unimodular-sequences} of \cref{assumption:R-2} implies that the link of every $k$-simplex in $\U_n(R)$ is non-empty for $k \leq n-2$. To finish, we record the following observation about our assumptions.

\begin{corollary}
	\label{corollary:assumption-consequence-2}
	If $R$ satisfies \cref{assumption:R} or \cref{assumption:R-2}, then $R$ is Hermite, $\bB_n(R) = \bU_n(R)$, and this complex is pure of dimension $(n-1)$.
\end{corollary}

\subsection{Linear algebra over Hermite rings}

In this section we collect basic facts about modules over Hermite rings (i.e.\ stably free $R$-modules are free) that we will frequently use in this work, similar to \cite[Section 2.2]{churchputman2017}.

\begin{lemma}
	\label{lemma:complements-are-free}
	Let $R$ be a Hermite ring, $M$ be a free $R$-module and $V$ be a free summand of $M$. Then every complement $C$ of $V$ in $M = V \oplus C$ is also a free summand of $M$. Equivalently, $C \cong M/V$ is free. If $R$ additionally satisfies \cref{ibn-I}, \cref{ibn-II} or \cref{ibn-III} and $\rank(M), \rank(V) < \infty$, then $\rank(C) = \rank(M) - \rank(V)$.
\end{lemma}
\begin{proof}
	This follows from the definition, the exact sequence $0 \to V \to M \to M/V \to 0$ and \cite[Corollary 0.4.2]{cohn2006}.
\end{proof}

We use this to check the following analogue of \cite[Lemma 2.6]{churchputman2017} for Hermite rings.

\begin{lemma}
	\label{lemma:summand-property}
	Let $R$ be a Hermite ring. If $V$ and $V'$ are free summands of $R^n$ such that $V \subseteq V'$, then $V$ is a free summand in $V'$.
\end{lemma}

\begin{proof}
	Let $R^n = V \oplus C$ and $R^n = V' \oplus C'$. \cref{lemma:complements-are-free} implies that $C \cong R^n/V$ and $C' \cong R^n/V'$ are free. This means that $0 \to V'/V \to R^n/V \to R^n/V' \to 0$ is split. Hence $V'/V$ is stably free and therefore free, because $R$ is Hermite. It therefore follows that $0 \to V \to V' \to V'/V \to 0$ is split and this yields the claim.
\end{proof}

\begin{lemma}
	\label{lemma:intersection-yields-complement}
	Let $R$ be Hermite and have the invariant basis number property. Consider two direct sum decompositions $R^n = V \oplus W = V' \oplus W'$ into free modules $V, W, V', W'$ such that $V \subseteq V'$ and $W' \subseteq W$. Then
	$V' \cap W$ is a free $R$-module of rank $(\rank(V') - \rank(V)) = (\rank(W) - \rank(W'))$ and 
	$$R^n = V \oplus (V' \cap W) \oplus W', \quad V' = V \oplus (V' \cap W), \quad W = (V' \cap W) \oplus W'.$$
\end{lemma}

\begin{proof}
	It suffices to check that $V' = V \oplus (V' \cap W)$ and $W = (V' \cap W) \oplus W'$, for then the claim follows from \cref{lemma:complements-are-free}. If $v' \in V' \subseteq R^n$, then there is a unique way of writing $v' = v + w$ for $v \in V$ and $w \in W$. Since $w = v' - v$, it holds that $w \in V' \cap W$. The first claim follows. The second claim can be checked similarly: If $w \in W \subseteq R^n$, then there is a unique way of writing $w = v' + w'$ for $v' \in V'$ and $w' \in W'$. Since $v' = w - w'$ it holds that $v' \in V' \cap W$.
\end{proof}

\subsection{The opposite ring and left-right duality}

Since we do not assume that $R$ is commutative or that $R$ admits an anti-automorphism, the rings $R$ and $R^{op}$ are a priori different and can not be identified with one another. Nevertheless, the following shows that if either satisfies \cref{assumption:R} then both do.

For the first two items in \cref{assumption:R}, this is well-known.

\begin{lemma}
	\label{lemma:relation-to-Rop}
	$R$ satisfies \cref{assumption-invariant-basis} and \cref{assumption-stably-free} of \cref{assumption:R} if and only if $R^{op}$ does.
\end{lemma}

\begin{proof}
	Assuming that $R$ satisfies \cref{assumption-invariant-basis} and \cref{assumption-stably-free} of \cref{assumption:R}, we check that $R^{op}$ also does. This suffices because $(R^{op})^{op} = R$. By \cite[Theorem 2.7]{cohn1966}, $R$ satisfies \cref{ibn-III}. It follows from \cite[Proposition 2.2]{cohn1966} that $R^{op}$ also satisfies \cref{ibn-III}, and hence also \cref{ibn-II} and \cref{ibn-I}, i.e.\ \cref{assumption-invariant-basis}. It follows from \cite[Theorem 0.4.1]{cohn2006} that $R$ satisfies \cref{assumption-stably-free} if and only if $R^{op}$ does (using that \cref{assumption-invariant-basis} and \cref{assumption-stably-free} together means $n$-Hermite for all $n \geq 1$ in the sense of \cite{cohn2006}).
\end{proof}

In \cref{appendix:partial-bases-of-the-opposite-ring}, we show that \cref{assumption-partial-bases} of \cref{assumption:R} also holds for $R^{op}$ if $R$ satisfies \cref{assumption:R}. This is surprising, because there does not seem to be a straightforward way to relate the partial bases complexes of $R^n$ and $(R^{op})^n$ if $R \ncong R^{op}$. As a consequence our argument is rather technical; it follows ideas developed by Sadofschi Costa in \cite{sadofschicosta2020}, and leads to the following theorem.

\begin{theorem}
	\label{theorem:relation-to-Rop}
	$R$ satisfies \cref{assumption:R} if and only if $R^{op}$ does.
\end{theorem}

\begin{proof}
	This follows from \cref{lemma:relation-to-Rop} and the results contained in \cref{appendix:partial-bases-of-the-opposite-ring}.
\end{proof}

This left-right duality result, i.e.\ \cref{theorem:relation-to-Rop}, is first and foremost a theoretical insight. It is important, because some of arguments in this work actually require that \emph{both} $R$ and $R^{op}$ satisfy \cref{assumption:R}: In the proofs of \cref{freeCharney} and \cref{SSTvanish} discussed in \cref{appendix:standard-connectivity-estimate-and-coinvariants-of-the-charney-module}, we carry out the analogue of an argument of Kupers--Miller--Patzt \cite{kupersmillerpatzt2022} in our setting, and use \cref{theorem:relation-to-Rop} to pass from $R$-modules $M$ to $R^{op}$-modules $M^\vee = \Hom_R(M, R)$. The key dualizing trick used in these proofs (compare with \cref{lemma:dualizing-argument-R-vs-Rop}) is due to Charney \cite{charney1980}. Left-right duality allows us to apply it without assuming that $R \cong R^{op}$ and without additional assumptions on $R^{op}$.

In practice, \cref{theorem:relation-to-Rop} was known to hold for all examples that the authors are aware of: If $R$ is a commutative ring (e.g.\ an Euclidean or Dedekind domain) or a ring with anti-automorphism, then $R \cong R^{op}$. If $R$ is a ring of stable rank one, a result of Vasterstein \cite[Theorem 2]{vaserstein1971} shows that the stable ranks of $R$ and $R^{op}$ are equal, $\on{sr}(R) = \on{sr}(R^{op})$. Therefore, van der Kallen's theorem \cite[§2. An Acyclicity Theorem]{vanderkallen1980} applies to both $R$ and $R^{op}$ to show that $\B_n(R)$ and $\B_n(R^{op})$ are Cohen--Macaulay. It is natural to ask if this is a coincidence, and \cref{theorem:relation-to-Rop} shows that it is not.

We close this section by recording an elementary observation related to left-right duality, which is also used in \cref{appendix:partial-bases-of-the-opposite-ring} and \cref{appendix:standard-connectivity-estimate-and-coinvariants-of-the-charney-module}.

\begin{lemma}
	\label{lemma:dualizing-and-inverse-transpose}
	Let $R$ be a unital ring and $C$ denote a free $R$-module $C$ of rank $n$. Then $C^\vee = \Hom_R(C, R)$ is a free $R^{op}$-module of rank $n$, and taking the inverse-transpose yields a group isomorphism 
\begin{align*} \GL(C) &\xrightarrow{\cong} \GL(C^{\vee})\\  \phi &\mapsto (\phi^{-1})^*,\end{align*}
	where $(\phi^{-1})^*$ is the automorphism that acts on a function in $C^\vee$ by precomposing with $\phi^{-1}$.
\end{lemma}

\begin{proof}
	Picking an $R$-module basis of $C$ leads to an isomorphism of rings $\End_R(C) \cong \on{Mat}_{n \times n}(R)$, and hence an isomorphism of the groups of units $\GL(C) \cong \GL_n(R)$. Similarly, the associated dual basis for $C^\vee$ shows that it is a free $R^{op}$-module of rank $n$, provides a ring isomorphism $\End_{R^{op}}(C^\vee) \cong \on{Mat}_{n \times n}(R^{op})$, and an isomorphism of groups of units $\GL(C^\vee) \cong \GL_n(R^{op})$. 
	Taking inverses $\phi \mapsto \phi^{-1}$ defines an isomorphism between the group of units of $\End_R(C)$ and $\End_R(C)^{op}$, $\GL(C)$ and $\GL(C)^{op}$. Acting by precomposition yields a ring isomorphism $\End_R(C)^{op} \to \End_{R^{op}}(C^\vee): \psi \mapsto \psi^*$, where $\psi^*(f) = f \circ \psi$ for $f \in \End_{R^{op}}(C^\vee)$. Hence, the map in the claim is an isomorphism $\GL(C) \to \GL(C^{\vee})$. Coordinate isomorphisms identify this isomorphism with the inverse-transpose isomorphism $\GL_n(R) \to \GL_n(R^{op}): M \mapsto (M^{-1})^T$, hence the name.
\end{proof}
	\section{(Relative) Steinberg modules}
	\label{sec:relative-steinberg-modules}
We introduce the Steinberg modules $\St_n(R)$ and the relative Steinberg modules $\St_n^m(R)$ following \cite[Section 4.2]{kupersmillerpatzt2022}. To do this we associate certain (relative) Tits complexes to $R$. The content of this section is closely related to recent work of Scalamandre \cite{scalamandre2023}.

\begin{definition}
	\label{def:tits-building}
	Let $M$ be a free $R$-module, and let $\bT(M)$ be the poset of nonzero proper \emph{free} summands of $M$, ordered by inclusion. Let $T(M)$ denote the geometric realization of $\bT(M)$. We will write $\bT_n(R)$ for $\bT(R^n)$ and $T_n(R)$ for its geometric relatization, and refer to $\bT_n(R)$ as the \emph{Tits complex}.
\end{definition}

The next lemma shows that, under \cref{assumption:R} and assuming that $R$ is commutative, the Tits complexes above agree with those recently defined and studied by Scalamandre \cite{scalamandre2023}. This is the reason why we named them exactly as in \cite{scalamandre2023}.

\begin{lemma}
	Under \cref{assumption:R}, the Tits complex in \cref{def:tits-building} agrees with the one defined by Scalamandre \cite[Definition 3.12]{scalamandre2023} (dropping the commutativity assumption).
\end{lemma}

\begin{proof}
	Dropping the commutativity assumption, the Tits complex $\bT_n^S(R)$ in \cite[Definition 3.12]{scalamandre2023} is the poset of nonzero proper summands $V$ of $R^n$ such that $V$ and $R^n/V$ are free $R$-modules, where $V \leq_S V'$ if $V' = V \oplus C$ with $C$ a free $R$-module. \cref{lemma:complements-are-free} implies that for every free summand $V \in T_n(R)$ the quotient $R^n/V$ is also free. Hence, $\bT_n^S(R)$ and $\bT_n(R)$ have the same underlying set. If $V, V' \in T_n(R)$ such that $V \subseteq V'$, then \cref{lemma:summand-property} implies that $V = V' \oplus C$ for some $C$ and another application of \cref{lemma:complements-are-free} shows that $C$ is free. Hence, it holds that $V \leq_S V'$.
\end{proof}

The proof of \cite[Proposition 2.6]{scalamandre2023} yields the following description of the $k$-simplices of $T_n(R)$ in our setting if one refers to \cref{lemma:complements-are-free} instead of \cite[Lemma 2.5]{scalamandre2023}.

\begin{lemma}
	\label{lemma:simplices-of-the-tits-complex}
	Under \cref{assumption:R}, the following are equivalent:
	\begin{enumerate}
		\item $V_0 \lneq \dots \lneq V_k$ is a $k$-simplex in $T_n(R)$;
		\item $V_0 \lneq \dots \lneq V_k$ is a face of a $(n-2)$-simplex in $T_n(R)$;
		\item There exist partial bases $\emptyset \neq \sigma_0 \subsetneq \dots \subsetneq \sigma_k$ of $R^n$ such that $V_i = \langle \sigma_i \rangle_R$;
		\item $V_0 \neq 0$, $V_i$ and $V_{i+1}/V_i$ are free for every $0 \leq i \leq k-1$, $V_k$ and $R^n/V_k$ are free.
	\end{enumerate}	
	In particular, $T_n(R)$ is a pure simplicial complex of dimension $(n-2)$.
\end{lemma}

\begin{remark}
	In general \cite[Lemma 2.5]{scalamandre2023} does not hold for non-commutative rings, see e.g.\ \cite[Theorem 4.11]{lam2006} and \cite[page 36, 2nd paragraph]{lam2006}.
\end{remark}

\begin{remark}
	For rings that are not PIDs, \cref{def:tits-building} differs from the usual definition of the Tits building as we require that the submodules are free. In particular, $T_n(R)$ is in general not a building in the sense of Brûhat--Tits. This is another reason why we chose to follow the naming convention in \cite{scalamandre2023}, compare \cite[Remark 1.4]{scalamandre2023}.
\end{remark}

We will also be interested in the following relative versions of the Tits complex.

\begin{definition}
	For $N$ a free summand of a free $R$-module $M$, let $\bT(M,N)$ be the subposet of $\bT(M)$ of summands of complements of $N$. We write $\bT_n^m(R)$ for the \emph{relative Tits complex} $\bT(R^{m+n},R^m)$, where here $R^m$ is viewed as a summand of $R^{m+n}$ via the standard decomposition $R^{m+n} = R^m \oplus R^n$. Let $T(M,N)$ and $T_n^m(R)$ denote the respective geometric realizations. 
\end{definition}

Note that under \cref{assumption:R} and if $m > 0$, the relative Tits complex $T_n^m(R)$ is $(n-1)$-dimensional as a consequence of \cref{lemma:simplices-of-the-tits-complex}. Next, we relate the (relative) Tits complexes to the partial bases complexes introduced in \cref{definition:partial-bases-poset}.

\begin{definition}
	\label{definition:complex-of-partial-basis-2}
	Let $M$ be a free $R$-module. The complex of partial bases $\B(M)$ is the
	simplicial complex where a collection of vectors $\{\vec v_0, \dots, \vec v_{k}\}$ in $M$ spans a $k$-simplex if it is a subset of a basis of $M$. Let $\B_n(R)$ denote $\B(R^n)$, and let $\B_n^m(R)$ denote $\Link_{\B(R^{n+m})}(\{\vec e_1,\ldots, \vec e_m\})$, where $\vec e_i$ is the i-th standard unit vector in $R^{m+n}$. We write $\bB(M), \bB_n(R)$ and $\bB_n^m(R)$ for the simplex posets of these simplicial complexes.
\end{definition}

There is an interesting map of posets between the simplex posets of (relative) partial bases and the (relative) Tits building, given by taking spans.

\begin{definition}
	\label{definition:span-maps}
	If $m = 0$, let $\bB_n(R)^{(n-2)}$ denote the $(n-2)$-skeleton of $\bB_n(R)$. Then there is a poset map $\Span \colon \bB_n(R)^{(n-2)} \m \bT_n(R)$ defined by sending a partial basis $\sigma$ to the $R$-linear summand $\langle \sigma \rangle$ of $R^n$ it spans. If $m > 0$, there is a similarly defined poset map $\Span \colon \bB_n^m(R) \m \bT_n^m(R)$ sending a partial basis $\sigma$ of a complement of $R^m$ in $R^{m+n}$ to the $R$-linear summand $\langle \sigma \rangle$ it spans.
\end{definition} 

The following is a consequence of \cref{lemma:simplices-of-the-tits-complex} and is the analogue of \cite[Lemma 3.14]{scalamandre2023} in our setting.

\begin{lemma}
	\label{lemma:surjectivity-of-span}
	Let $m > 0$. Under \cref{assumption:R}, $\Span \colon \bB_n(R)^{(n-2)} \m \bT_n(R)$ and $\Span \colon \bB_n^m(R) \m \bT_n^m(R)$ are surjective poset maps.
\end{lemma}

Using standard poset arguments together with the high connectivity of $\bB_n^m(R)$, i.e.\ \cref{assumption-partial-bases} of \cref{assumption:R}, we can deduce that $\bT_n^m(R)$ is highly connected as well. This observation has been used by Kupers--Miller--Patzt in \cite[Lemma 4.6]{kupersmillerpatzt2022}, and the proof of the following lemma is completely analogous to their argument.

\begin{lemma}
	\label{lemma:span-map}
	Suppose \autoref{assumption:R} holds and let $m > 0$. Then
	\begin{enumerate}
		\item \label{lemma:span-map-item-1} the spanning map $\Span\colon \bB_n(R)^{(n-2)} \to \bT_n(R)$ is $(n-2)$-connected;
		\item \label{lemma:span-map-item-2} the spanning map $\Span\colon \bB^m_n(R) \to \bT^m_n(R)$ is $(n-1)$-connected.
	\end{enumerate}
	Thus, $T_n(R)$ is Cohen--Macaulay and, if $m > 0$, then $T_n^m(R)$ is $(n-1)$-spherical.
\end{lemma}

The fact that the Tits complex $\bT_n(R)$ is Cohen--Macaulay was previously proved by Scalamandre under slightly different assumptions, see \cite[Theorem 4.2, Case $m = n-1$]{scalamandre2023}. Scalamandre's result generalizes a classical theorem due to Solomon--Tits \cite{solomon1969} for fields $R$ as well as unpublished work of Rognes \cite{rognes1991} for $R = \Z/p^n$. The following includes an alternative proof in our setting, based on a connectivity theorem due to van der Kallen--Looijenga \cite[Corollary 2.2]{vanderkallenlooijenga2011} and an argument contained in the proof of \cite[Lemma 4.6]{kupersmillerpatzt2022}.

\begin{proof}
	As in the proof of \cite[Lemma 4.6]{kupersmillerpatzt2022}, we argue by induction on $n$ using \cref{assumption:R} and \cite[Proposition 4.1]{kupersmillerpatzt2022} (this is \cite[Corollary 2.2]{vanderkallenlooijenga2011}). Recall that we already established that $\Span$ is a surjective poset map.
	
	For \cref{lemma:span-map-item-1}: We start by noting that $\bB_n(R)^{(n-2)}$ is Cohen--Macaulay of dimension $(n-2)$ by \cref{assumption:R} and e.g.\ the discussion on \cite[page 1887]{hatchervogtmann2017}. If $n = 0$ or $n = 1$ the claim is void. For $n = 2$, both complexes are discrete sets and $\bT_n(R)$ is the set of free rank-$1$ summands $L$ in $R^2$ that have a free complement. Since the span map is surjective, it is surjective on $\pi_0$ if $n = 2$, i.e.\ $\Span$ is $0$-connected. Assume that $n > 2$ and let $V \in \bT_n(R)$. Then $\Span_{\leq V} = \bB(V)$, and $\bT_n(R)_{> V} \cong \bT_{n-\rank(V)}(R)$ by \cref{upperlowerTits}. $\bB(V)$ is $(\rank(V) - 2)$-connected by \cref{assumption:R}, and the induction hypothesis implies that $\bT_{n-\rank(V)}(R)$ is $(n - \rank(V) - 3)$-connected since $n - \rank(V) < n$. It follows that the assumptions \cite[Proposition 4.1, Part 2]{kupersmillerpatzt2022} for proving that $\Span$ is $(n-2)$-connected are satisfied if we set $t(V) \coloneqq \rank(V) - 1$ for $V \in \bT_n(R)$. This finishes the proof of \cref{lemma:span-map-item-1}, and shows that $T_n(R)$ is $(n-3)$-connected. The Cohen--Macaulay property for $T_n(R)$ follows from this, \cite[Proposition 8.6]{quillen1978} and \cref{upperlowerTits} exactly as in the last paragraph of \cite[page 25, Proof of Theorem A]{scalamandre2023}.
	
	For \cref{lemma:span-map-item-2}, the argument is similar: If $n = 0$ the claim is void. If $n = 1$ the domain and codomain are discrete posets, and the claim follows by direct inspection and the surjectivity of the poset map $\Span$ exactly as in the first case. If $n > 1$, we consider $V \in \bT^m_n(R)$, again set $t(V) \coloneqq \rank(V) - 1$, and observe that $\Span_{\leq V} = \bB(V)$ and  $\bT^m_n(R)_{> V} \cong \bT^m_{n -\rank(V)}(R)$ by \cref{upperlowerRelTits}. The assumptions of \cite[Proposition 4.1, Part 2]{kupersmillerpatzt2022} for proving that $\Span$ is $(n-1)$-connected again follow from the induction hypothesis and \cref{assumption:R}.
\end{proof}

Given a free summand $N$ of an $R$-module $M$, we denote by $\GL(M, N) \leq \GL(M)$ the subgroup that fixes $N$ pointwise. If $M = R^{m+n}$ and $N = R^n$, we write $\GL_n^m(R) = \GL(R^{m+n}, R^m)$. We note that $\GL_n(R)$ acts (from the right) on $T_n(R)$ simplicially, and similarly $\GL_n^m(R)$ acts on $T_n^m(R)$. As a consequence of \cref{lemma:span-map}, $T_n(R) \simeq \vee S^{n-2}$ and $T_n^m(R) \simeq \vee S^{n-1}$ for $m > 0$ and if $R$ satisfies \cref{assumption:R}. In particular, the complexes have a single (possibly) nontrivial reduced homology group in degree $(n-2)$ and $(n-1)$. This leads us to the following definition.

\begin{definition}
	\label{definition:relative-steinberg-modules}
	Let $R$ be a ring such that $T_n(R)$ is $(n-2)$-spherical. Then the \emph{Steinberg module} of $R$, denoted by $\St_n(R)$, is the right $\GL_n(R)$-module $\widetilde{H}_{n-2}(T_n(R); \Z)$. If $m>0$ and $T_n^m(R)$ is $(n-1)$-spherical, then the \emph{relative Steinberg module} $\St^m_n(R)$ is the right $\GL_n^m(R)$-module $\widetilde{H}_{n-1}(T_n^m(R);\Z)$.
\end{definition}
	\section{Generating sets for (relative) Steinberg modules}
	\label{sec:generating-sets-for-relative-steinberg-modules}
To obtain our stability result following \cite{kupersmillerpatzt2022}, we need to prove a vanishing result for certain co-invariants of the (relative) Steinberg modules introduced in \cref{definition:relative-steinberg-modules}.

For the non-relative Steinberg modules $\St_n(R)$, \cref{assumption:R} allows one to run an argument due to Church--Farb--Putman \cite{churchfarbputman2019} to obtain a generating set for $\St_n(R)$, which can be used to show that $(\St_n(R) \otimes \Z[1/2])_{\GL_{n}(R)} = 0$ if $n \geq 2$. Recently, Scalamandre \cite{scalamandre2023} carried out this argument in the generality that we require in this work. For Euclidean domains the construction of this generating set is due to Ash--Rudolph \cite[Theorem 4.1]{ashrudolph1979} and the vanishing result is due to Lee--Szczarba \cite[Theorem 1.3]{leeszczarba1976}. 

The main technical innovation of this article is to show that, if $m > 0$,  \cref{assumption:R} also suffices to construct a generating set for the relative Steinberg modules $\St_n^m(R)$, which can be used to check that $(\St_n^m(R) \otimes \Z[1/2])_{\GL_{n}^m(R)} = 0$ if $n \geq 1$. Compared to our \cref{assumption-partial-bases} in \cref{assumption:R}, the arguments in \cite[Theorem 3.15 and Lemma 4.6]{kupersmillerpatzt2022} require much more complicated connectivity assumptions, which to date are only known to hold for fields, the intergers \cite{churchputman2017}, the Gaussian and the Eisenstein integers \cite{kupersmillerpatztwilson2022}. In contrast, \cref{assumption:R} is known to hold for all local, semi-local or Artinian rings, Euclidean domains as well as certain non-euclidean Dedekind domains (compare \cref{section-1}). To prove our result, we introduce a modified partial basis complex, which is the key ingredient in a new induction procedure that we use to construct a generating set for $\St_n^m(R)$ with $m > 0$.

\subsection{Apartment classes and the Steinberg module}

Let \cref{assumption:R} hold and $\partial \Delta_{n-1}$ be the boundary of the standard $(n-1)$-simplex. Its poset of simplices can be identified with the poset $\bT(\llbracket n \rrbracket)$ of nonempty proper subsets of $\llbracket n \rrbracket \coloneqq \{1, \dots, n\}$. For any given matrix $M \in \GL_n(R)$ with column vectors $M = [\vec M_1 \dots \vec M_n]$, we obtain a poset embedding
\[
M: \bT(\llbracket n \rrbracket) \hookrightarrow \bT_n(R): S \mapsto \langle \vec M_i : i \in S \rangle.
\]
Fixing a generator $\eta_{n-2} \in \widetilde{H}_{n-2}(\bT(\llbracket n \rrbracket))$ once and for all, we obtain a unique class $[M] \coloneqq M_*(\eta_{n-2}) \in \widetilde{H}(\bT_n(R); \Z) = \St_n(R)$, which is called the \emph{apartment class} of $M \in \GL_n(R)$.

\begin{theorem}
	\label{theorem:apartment-classes-generate-steinberg}
	Let $R$ be a ring such that \cref{assumption:R} holds. Then the map
	\[
		\Z[\GL_{n}(R)] \to \St_n(R): M \mapsto [M]
	\]
	is a $\GL_n(R)$-equivariant surjection. In particular, $\St_n(R)$ is generated as an abelian group by the set apartment classes $\{[M]: M \in \GL_n(R)\}$.
\end{theorem}

\begin{proof}
	This follows from an argument due to Church--Farb--Putman \cite[Proof of Theorem A]{churchfarbputman2019}. It was carried out by Scalamandre in the required generality in \cite[Theorem 5.1]{scalamandre2023}, which we can apply to prove the claim using \cref{assumption:R} and \cref{lemma:span-map}.
\end{proof}

\begin{corollary}
	\label{corollary:coinvariants-of-steinberg}
	Let $R$ be a ring such that \autoref{assumption:R} holds and let $\K$ be a commutative ring in which $2 \in \K^\times$ is a unit. If $n \geq 2$, then the $\GL_{n}(R)$-coinvariants of $\St(R) \otimes \K$ vanish. In symbols,
	\[
		(\St_n(R) \otimes \K)_{\GL_n(R)} = 0.
	\]
\end{corollary}
\begin{proof}
	We can apply a trick due to Church--Farb--Putman \cite[Proof of Theorem C]{churchfarbputman2019}: Let $n \geq 2$ and let $[M] \in \St_n(R)$ be an apartment class where $M = [\vec M_1 \dots \vec M_n]$. By virtue of \cref{theorem:apartment-classes-generate-steinberg}, we only need to show that $[M] \otimes 1 = 0$ in $(\St_n(R) \otimes \K)_{\GL_n(R)}$. To see this, consider the element $\phi \in \GL_n(R)$ that maps $\vec M_1 \mapsto \vec M_2$, $\vec M_2 \mapsto \vec M_1$ and $\vec M_i \mapsto \vec M_i$ for $i > 2$. The element $\phi$ acts by reversing the orientation of $[M]$, i.e.\ $[M] \cdot \phi = - [M]$, and therefore in the coinvariants it holds that $[M] \otimes 1 = - [M] \otimes 1$ and, since $2$ is a unit in $\K$ and $2 \cdot([M] \otimes 1) = 0$, we must have that $[M] \otimes 1 = 0$.
\end{proof}

\subsection{Relative apartment classes and relative Steinberg modules}

Now we define the set that will parametrize generators of the relative Steinberg module, as well as modified versions of this set that will be needed in the proof.

\begin{definition} 
	Let $m,n \geq 1$.
	\begin{enumerate}
		\item When $0\leq j\leq n$, we write $\relstgen^m_n(j)$ for the set of formal symbols 
		\[
			[\vec v_1, \vec v_1 + r_1 \vec e_{\beta(1)}] \ast \dots \ast [\vec v_j, \vec v_j + r_j \vec e_{\beta(j)}]
		\]
		obtained from a choice of ordered simplex $[\vec v_1, \dots, \vec v_j]$ of $\B^m_n(R)$, a choice of coefficients $r_i\in R - \{0\}$, and a choice of vectors $\vec e_{\beta(i)} \in \{\vec e_1, \dots, \vec e_m,\vec v_1,\dots, \vec v_{i-1}\}$. We use the conventions that	$\relstgen^m_n(0)$ contains a single element (thought of as an empty simplex in $\B^m_n(R)$), and when $j=n$, we write $\relstgen^m_n$ for $\relstgen^m_n(n)$.
		\item We define $\Z\{\relstgen^m_n\}$ to be the right $\GL^m_n(R)$-module with underlying $\Z$-module
		\[
			\Z\{\relstgen^m_n\} \coloneqq \bigoplus_{\relstgen^m_n} \Z,
		\]
		and whose $\GL^m_n(R)$-action on the basis elements $\relstgen^m_n$ is defined by
		\begin{multline*}
			([\vec v_1, \vec v_1 + \alpha_1 \vec e_{\beta(1)}] \ast \dots \ast [\vec v_n, \vec v_n + \alpha_n \vec e_{\beta(n)}]) \cdot \phi =\\
			[\phi(\vec v_1), \phi(\vec v_1) + \alpha_1 \phi(\vec e_{\beta(1)})] \ast \dots \ast [\phi(\vec v_n), \phi(\vec v_n) + \alpha_n \phi(\vec e_{\beta(n)})]
		\end{multline*}
		for $\phi \in \GL^m_n(R)$.
	\end{enumerate}
\end{definition}

Let $R$ be a ring such that \autoref{assumption:R} is satisfied and $m, n \geq 1$. We will prove that there exists a $\GL^m_n(R)$-equivariant surjection
$$F \colon \Z \{\relstgen^m_n\} \twoheadrightarrow \widetilde{H}_{n-1}(\bT^m_n;\Z) = \St_n^m(R).$$
This map is constructed as follows. Consider a symbol 
\[
	\Theta = [\vec v_1,\vec v_1+\alpha_1 \vec e_{\beta(1)}]\ast\dots\ast [\vec v_n,\vec v_n+\alpha_n \vec e_{\beta(n)}] \in \relstgen^m_n.
\]
Let $1_\Theta \in \Z$ denote $1$ in the summand of the domain of $F$ indexed by $\Theta$. Let 
\[
	S^{n-1} = S^0 \ast ... \ast S^0
\] 
be the simplicial $(n-1)$-sphere obtained as a join of $n$ copies of $S^0$, and denote its poset of simplices by $\mathbb{S}^{n-1}$. Each symbol $\Theta \in \relstgen^m_n$ gives rise to a poset embedding
\[
	\Theta \colon \mathbb{S}^{n-1} \hookrightarrow \bT^m_n(R)
\]
by mapping the two vertices of the i-th copy of $S^0$ in the join $S^{n-1} = S^0 \ast ... \ast S^0$ to $\{\langle\vec v_i \rangle, \langle \vec v_i+\alpha_i \vec e_{\beta(i)} \rangle\}$. Passing to homology we obtain a map
\[
	\Theta_* \colon \widetilde{H}_{n-1}(\mathbb{S}^{n-1}; \Z) \to \widetilde{H}_{n-1}(\bT^m_n;\Z) = \St_n^m(R).
\]
Fix a generator $\eta_{-1} \in \widetilde{H}_{-1}(\mathbb{S}^{-1}; \Z) = \Z$. This choice determines generators $\eta_{n-1} \in \widetilde{H}_{n-1}(\mathbb{S}^{n-1}; \Z)$ for all $n \in \N$ using the suspension isomorphisms coming from the identifications $\mathbb{S}^0 \ast \mathbb{S}^{n-2} \cong \mathbb{S}^{n-1}$. The map $F$ is then defined by the formula
\[
	F(1_\Theta) = \Theta_*(\eta_{n-1}) \eqqcolon [\Theta].
\]
The class $[\Theta] \in \widetilde{H}_{n-1}(\bT_n^m(R); \Z) = \St^m_n(R)$ is called the \emph{relative apartment class} associated to the symbol $\Theta \in \relstgen^m_n$. Our main technical result is the following theorem. 

\begin{theorem}\label{theorem:generators}
	Let $R$ be a ring such that \autoref{assumption:R} holds and $m, n \geq 1$. Then
	$$F \colon \Z \{\relstgen^m_n\} \twoheadrightarrow \St_n^m(R): \Theta \mapsto [\Theta]$$
	is a $\GL^m_n(R)$-equivariant surjection. In particular, the set of relative apartment classes $\{[\Theta] : \Theta \in \relstgen^m_n\}$ generates $\St^m_n(R)$ as an abelian group.
\end{theorem}

This generating set allows us to obtain the following vanishing result, which is the key input for the homological stability theorems proved in the next section.

\begin{corollary}\label{corollary:coinvariants-of-relative-steinberg}
	Let $R$ be a ring such that \autoref{assumption:R} holds, and let $\K$ be a commutative ring in which $2 \in \K^\times$ is a unit.
	If $n,m \geq 1$, then the $\GL_n^m(R)$-coinvariants of $\St_n^m(R) \otimes \mathbb{k}$ vanish. In symbols,
	$$(\St_n^m(R) \otimes \mathbb{k})_{\GL_n^m(R)} = 0 \text{ if } m, n \geq 1. $$
\end{corollary} 
	
	\begin{proof}[Proof of \autoref{corollary:coinvariants-of-relative-steinberg}]
		This argument is inspired by a trick used in \cite[Theorem 3.15, Case 2, and Proposition 3.16.]{kupersmillerpatzt2022}. By \cref{theorem:generators} it suffices to show every generator $[\Theta] \otimes 1 \in (\St^m_n(R) \otimes \K)_{\GL_n^m(R)}$ is equal to zero if 2 is a unit in $\K$. Let
		\[
			\Theta = [\vec v_1,\vec v_1+\alpha_1 \vec e_{\beta(1)}]\ast\dots\ast [\vec v_n,\vec v_n+\alpha_n \vec e_{\beta(n)}] \in \relstgen^m_n.
		\]
		Then $\{\vec e_1, \dots, \vec e_m, \vec v_1, \dots, \vec v_n\}$ is a basis of $R^{m+n}$ and $\vec e_{\beta(n)} \in \{\vec e_1, \dots, \vec e_m, \vec v_1, \dots, \vec v_{n-1}\}$.  Consider the element $\phi$ of $\GL_n^m(R)$ that sends $\vec v_n$ to $-\vec v_n -\alpha_n \vec e_{\beta(n)}$ and is the identity on $\vec e_i$ and $\vec v_j$ when $j \not= n$.
		It follows that
		\[
			\Theta \cdot \phi =  [\vec v_1,\vec v_1+\alpha_1 \vec e_{\beta(1)}] \ast\dots\ast [\vec v_{n-1},\vec v_{n-1}+\alpha_{n-1} \vec e_{\beta(n-1)}] \ast [- (\vec v_n + \alpha_n \vec e_{\beta(n)}), - \vec v_n].
		\]
		Observe that this implies $[\Theta] = - [\Theta \cdot \phi] = - [\Theta] \cdot \phi \in \St_n^m(R)$.		Therefore, in the coinvariants it holds that
		\[
			[\Theta] \otimes 1 = - ([\Theta] \cdot \phi \otimes 1) = - ([\Theta] \cdot \phi \otimes 1 \cdot \phi) = - ([\Theta] \otimes 1)
		\]
		and it follows that $2 \cdot ([\Theta] \otimes k) = 0$ in $(\St^m_n(R) \otimes \K)_{\GL_n^m(R)}$. Since 2 is a unit in $\K$, this implies that	$[\Theta] \otimes 1 = 0$ in $(\St^m_n(R) \otimes \K)_{\GL_n^m(R)}$ as claimed. \qedhere		
	\end{proof}
	
	Before proving \cref{theorem:generators}, we will introduce and study the connectivity properties of several simplicial complexes and posets.
	
	\subsubsection{Complex of externally augmented partial bases}
	
	The following new complex will play an important role in our arguments.
	
	\begin{definition}
		Fix a ring $R$ and an integer $m \geq 1$, and let $\gamma$ be either a simplex in
		$\B^m_n(R)$ or the empty set. Let $\BX^{m, \gamma}_n(R)$ denote the simplicial
		complex whose vertices are vectors $\vec v$ in $R^{m+n}$ with the
		property that $\{\vec e_1, \dots, \vec e_m\} \sqcup \gamma \sqcup \{\vec v\}$ is
		a partial basis of $R^{m+n}$ (i.e.\ a simplex in $B_{m+n}(R)$). The simplicial complex $\BX^{m, \gamma}_n(R)$ has two different types of simplices:
		\begin{enumerate}
			\item \emph{standard} k-simplices $\sigma = \{\vec v_1, \dots, \vec v_{k+1}\}$
			defined by the property that $\{\vec e_1, \dots, \vec e_m\} \sqcup \gamma \sqcup
			\{ \vec v_1, \dots, \vec v_{k+1}\}$ is a partial basis of size
			$m+|\gamma|+(k+1)$ of $R^{m+n}$;
			\item \label{item-externally-additive} \emph{externally additive} k-simplices $\sigma = \{\vec v_1 = \vec
			v_2 + r \vec e_{\beta(2)}, \vec v_2, \dots, \vec v_{k+1}\}$ defined by the
			property that $\{\vec e_1, \dots, \vec e_m\} \sqcup \gamma \sqcup \{\vec v_2,
			\dots, \vec v_{k+1}\}$ is a partial basis of size $m+|\gamma|+k$ of $R^{m+n}$
			and that $\vec v_1$ is a sum of the form $\vec v_2 + r \vec e_{\beta(2)}$ for
			$e_{\beta(2)} \in \{\vec e_1, \dots, \vec e_m\} \sqcup \gamma$ and $r \in
			R\setminus \{0\}$.
		\end{enumerate}
		We will refer to the simplicial complex $\BX^{m, \gamma}_n(R)$ as the
		\emph{complex of externally augmented partial bases}. If $\gamma =
		\emptyset$, we write $\BX^{m}_n(R)$ for $\BX^{m, \gamma}_n(R)$.
	\end{definition}

	We highlight that $\BX^{m, \gamma}_n(R)$ is exactly $\Link_{\B_n^m(R)}(\gamma)$ with extra externally additive simplices, as in \cref{item-externally-additive}, glued on. In particular, $\Link_{\B_n^m(R)}(\gamma)$ is a subcomplex of $\BX^{m, \gamma}_n(R)$ and, if $\gamma= \emptyset$, $\B_n^m(R)$ is a subcomplex of $\BX^{m}_n(R)$.
	
	\begin{remark}
		The definition of the simplicial complexes $\BX^{m, \gamma}_n(R)$ is inspired by work of Church--Putman \cite{churchputman2017}. In their work on the rational codimension-one cohomology of $\SL_n(\Z)$, a closely related simplicial complex $\BA^m_n(\Z)$ of augmented partial frames plays a key role. The complex $\BA^m_n(\Z)$ is constructed using externally additive augmentations for which the $e_i$-coefficient satisfies $r = 1$, as well as internally additive augmentations, which do not occur in our setting. Similar complexes of augmented frames, $\BA^m_n(\Z[i])$ and $\BA^m_n(\Z[\rho])$, for the Gaussian integers and the Eisenstein integers have been studied in Kupers--Miller--Patzt--Wilson's work \cite{kupersmillerpatztwilson2022} on the rational codimension-one cohomology of $\SL_n(\Z[i])$ and $\SL_n(\Z[\rho])$. These three spherical complexes of augmented partial frames, $\BA^m_n(\Z)$, $\BA^m_n(\Z[i])$, and $\BA^m_n(\Z[\rho])$, have also been used by Kupers--Miller--Patzt \cite{kupersmillerpatzt2022} to construct generating sets for the relative Steinberg modules $\St^m_n(\Z)$, $\St^m_n(\Z[i])$ and
		$\St^m_n(\Z[\rho])$. The complexes $\BX^{m, \gamma}_n(R)$ of externally augmented partial bases occurring in this work are, in general, not spherical (see \cref{NotSpherical}). One way to think about the complexes
		$\BX^{m, \gamma}_n(R)$ is that they are similar to the complexes $\BA_n^m(R)$ except they do not have simplicies corresponding to the Bykovski\u{\i}-relations in Steinberg modules \cite{bykovskii2003}. These simplices turn out not to be relevant for slope-$1$ homological stability, and ignoring them allows us to prove our results for rings where the Bykovski\u{\i}-relations do not hold. 
	\end{remark}	

\subsubsection{Key observation}
	The next lemma will be key for establishing a relation between the span map $\Span\colon \bB^m_n(R) \to \bT^m_n(R)$ in \cref{definition:span-maps} and the relative apartment class map $F \colon \Z \{\relstgen^m_n\} \twoheadrightarrow \St^m_n(R)$ occurring in \cref{theorem:generators}. It studies a collection of $\Theta$-dependent span maps for each symbol $\Theta \in \relstgen^m_n(j)$. For $j = 0$ and $\Theta \in \relstgen^m_n(0)$ the empty symbol, the associated collection contains exactly one spanning map, which is $\Span\colon \bB^m_n(R) \to \bT^m_n(R)$.
	
	First we must construct these maps. Suppose \cref{assumption:R} holds and $m, n \geq 1$. Let $0 \leq j < n$ and consider a symbol
	$$\Theta = [\vec v_1,\vec v_1+\alpha_1 \vec e_{\beta(1)}]\ast\dots\ast [\vec v_j,\vec v_j+\alpha_j \vec e_{\beta(j)}] \in \relstgen^m_n(j).$$
	Associated with $\Theta$, we find a simplex $\sigma = \sigma(\Theta) = \{\vec v_1, \dots, \vec v_j\} \in \B^m_n(R)$.
	By restricting the spanning map
	$$\Span \colon \bB_{n}^{m}(R)\to \bT_n^m(R)$$
	to the simplex poset of $\Link_{\B^m_n}(\sigma)$, we obtain a map
	$$\bLink_{\B^m_n}(\sigma)\to \bT_n^m(R).$$
	Let $S^{j-1} = S^0 \ast \dots \ast S^0$ be a join of j copies of $S^0$. Using
	the symbol $\Theta$ we can define an embedding
	$$\Theta\colon \bS^{j-1} \hookrightarrow \bT_n^m(R),$$
	whose value on the two vertices of the i-th $S^0$ in $S^{j-1}$ is given by
	$\{ \langle \vec v_i \rangle, \langle \vec v_i+\alpha_i \vec e_{\beta(i)} \rangle \}$. The two maps are compatible
	and yield the \emph{$\Theta$-dependent span map}
	$$\Span_{\Theta} \colon \bS^{j-1} \ast \bLink_{\B^m_n}(\sigma) \to \bT_n^m(R).$$
	Consider any decomposition $R^{m + n} = R^m \oplus \langle \sigma \rangle \oplus C_\sigma$. Then $C_\sigma$ is free by \cref{lemma:complements-are-free}. Viewing $\B(C_\sigma)$ as a subcomplex of $\Link_{\B^m_n}(\sigma)$ via the inclusion
	$$C_\sigma \hookrightarrow R^{m+n}=R^m \oplus \langle \sigma \rangle \oplus C_\sigma$$
	yields a map
	$$\incl_{(\theta, C_\sigma)} \colon \bS^{j-1} \ast \bB(C_\sigma) \to \bS^{j-1} \ast \bLink_{\B^m_n}(\sigma).$$
	The first key observation is that the composition of this map with $\Span_{\Theta}$ is nullhomotopic.
	
	\begin{lemma}\label{lemma:BvsT} 
		If \cref{assumption:R} holds, $m,n \geq 1$ and $0\leq j<n$, then the composition
		$$\Span_{\Theta} \circ \incl_{(\theta, C_\sigma)} \colon \bS^{j-1} \ast \bB(C_\sigma) \to \bS^{j-1} \ast \bLink_{\B^m_n}(\sigma) \to \bT_n^m(R)$$
		is nullhomotopic for every choice of $C_\sigma$.
	\end{lemma} 
	
	\begin{proof}
		Recall that $C_\sigma \neq 0$ is a nonzero free summand of the complement $\langle \sigma \rangle \oplus C_\sigma$ of $R^m$ by \cref{lemma:complements-are-free} and therefore defines an element
		$$C_\sigma \in \bT^m_n(R).$$
		Since we view $\bB(C_\sigma)$ as a subposet of $\bLink_{\B^m_n}(\sigma)$ via the inclusion $C_\sigma \hookrightarrow R^{m+n}=R^m \oplus \langle \sigma \rangle \oplus C_\sigma$, it follows that for any partial basis $\tau \in
		\bB(C_\sigma)$ it holds that 
		$$\langle \tau \rangle \subseteq C_\sigma \text{ in } \bT^m_n(R).$$
		Hence, the image of $\Span_\Theta \circ \incl_{(\Theta, C_\sigma)}$ is contained in subspace $\Theta(\bS^{j-1}) \ast \bT^m_n(R)_{\leq C_\sigma},$ where $\bT^m_n(R)_{\leq C_\sigma}$ is the subposet of elements that are bounded above by $C_\sigma$. Since $\bT^m_n(R)_{\leq C_\sigma}$ has a cone point $C_\sigma$, it follows that the image of $\Span_\Theta \circ \incl_{(\Theta, C_\sigma)}$ is contained in a contractible subspace. Therefore, $\Span_\Theta \circ \incl_{(\Theta, C_\sigma)}$ is nullhomotopic.
	\end{proof}

\subsubsection{Relative connectivity results}
	
	In the next two steps, we investigate the connectivity of the complex $\BX^m_n(R)$, first relative to
	$\B^m_n(R)$ and then relative to $\B_n(R)$.
	
	\begin{lemma}\label{lemma:BXvsBrel}
		Suppose \autoref{assumption:R} holds. If $m \geq 1$, then $\pi_k(\BX^m_n(R), \B^m_n(R))$ and $H_k(\BX^m_n(R), \B^m_n(R))$ are trivial for all $k < n$.
	\end{lemma}
	
	\begin{proof}
		If $n = 0$, both $\B^m_n(R)$ and $\BX^m_n(R)$ are empty spaces, so there is nothing to prove. If $n = 1$, then $\B^m_1(R)$ is a discrete set, and we need to check that $H_0(\BX^m_1(R), \B^m_1(R)) = 0$. Since $\BX^m_1(R)$ and $\B^m_1(R)$	have the same vertex set, it holds that $\widetilde{H}_0(\B^m_1(R)) \twoheadrightarrow \widetilde{H}_0(\BX^m_1(R))$ is a surjection and this follows from the long exact homology sequence.
		
		For $n \geq 2$, we know that $\B^m_n(R)$ is $(n-2)$-connected and $n-2 \geq 0$. Therefore, $\B^m_n(R)$ is path-connected, and since $\B^m_n(R)$ is a subcomplex of $\BX^m_n(R)$ containing all vertices of $\BX^m_n(R)$, it follows that $\BX^m_n(R)$ is path-connected as well. We will prove that $\pi_k(\BX^m_n(R), \B^m_n(R)) = 0$ for $0 < k < n$. Since $\B^m_n(R)$ and $\BX^m_n(R)$ are path-connected, the relative Hurewicz theorem will then imply that $H_k(\BX^m_n(R), \B^m_n(R)) = 0$ for $k < n$.
		
		Let $\phi\colon (D^{k}, S^{k-1}) \to (\BX^m_n(R), \B^m_n(R))$ be a simplicial map, where $D^k$ is a disc of dimension $0 < k < n$ equipped with some combinatorial simplicial structure (see e.g.\ \cite[Definition 6.3 and Lemma 6.4]{putman2009}). Let $\Delta \in D^k$ be a simplex of maximal dimension with the property that $\phi(\Delta) = \{ \vec v_1 = \vec v_2 + r \vec e_i, \vec v_2\}$ is an externally additive edge in $\BX^m_n(R)$. Since $D^k$ is a combinatorial disc, it holds that $\Link_{D^k}(\Delta) \cong S^{k-\dim{\Delta}-1}$. The maximality assumption on $\Delta$ implies that $\phi(\Link_{D^k}(\Delta)) \subset \Link_{\BX^m_n(R)}(\phi(\Delta))$. We observe that $\Link_{\BX^m_n(R)}(\phi(\Delta)) \cong \Link_{\B^m_n(R)}(\{\vec v_2\}) \cong \B^{m+1}_{n-1}(R)$ is $(n-2)$-spherical. Since $k < n$ and $\dim \Delta \geq \dim \phi(\Delta) = 1$, it follows that $k-\dim{\Delta}-1 < n - 2$. Therefore, we see that the restriction
		$$ \phi|_\Link \colon \Link_{D^k}(\Delta) \to \Link_{\BX^m_n(R)}(\phi(\Delta))$$
		is nullhomotopic via a map
		$$ \psi \colon D(\Delta) \to \Link_{\BX^m_n(R)}(\phi(\Delta)),$$
		where $D(\Delta)$ is a combinatorial $(k-\dim{\Delta})$-disc (using e.g.\ \cite[Lemma 6.4]{putman2009}).
		We obtain a map 
		$$ \phi|_\Delta \ast \psi \colon \Delta \ast D(\Delta) \to \Star_{\BX^m_n(R)}(\phi(\Delta)) = \phi(\Delta) \ast \Link_{\BX^m_n(R)}(\phi(\Delta)),$$
		whose codomain is a contractible complex. Hence, the two maps
		$$ \phi|_{\partial \Delta} \ast \psi \colon \partial \Delta \ast D(\Delta) \to \Star_{\BX^m_n(R)}(\phi(\Delta))$$
		and
		$$ \phi|_{\Star} = \phi|_\Delta \ast \phi|_\Link \colon \Star_{D^k}(\Delta) = \Delta \ast \Link_{D^k}(\Delta) \to \Star_{\BX^m_n(R)}(\phi(\Delta))$$
		are homotopic relative to their restriction to the boundary $\partial \Delta \ast \Link_{D^k}(\Delta)$ of their domains. Using this homotopy we obtain a map $\phi'$ that is
		homotopic to $\phi$ and with the following description: The domain of $\phi'$ is a combinatorial disc $(D^k)'$ (using e.g.\ \cite[Lemma 5.7]{brueckpatztsroka2023}) that is obtained from $D^k$ by cutting out the combinatorial $k$-disc $\Star_{D^k}(\Delta) = \Delta \ast \Link_{D^k}(\Delta)$ and gluing in the combinatorial $k$-disc $\partial \Delta \ast D(\Delta)$ along the $(k-1)$-sphere $\partial(\partial \Delta \ast D(\Delta)) = \partial \Delta \ast \Link_{D^k}(\Delta)$. On $(\partial \Delta \ast D(\Delta), \partial \Delta \ast	\Link_{D^k}(\Delta))$, the map $\phi'$ is defined by $(\phi|_{\partial \Delta}	\ast \psi, \phi|_{\partial \Delta} \ast \phi|_\Link)$. Notice that the only	simplices in $\partial \Delta \ast D(\Delta)$ that $\phi|_\Delta \ast \psi$ can map to an externally additive edge are contained in $\partial \Delta$. Therefore, $\phi'$ maps one fewer simplex of dimension $\dim \Delta$ to an external edge than $\phi$ does. Recall that any simplex contained in the boundary sphere of the domain of $\phi$ is mapped to $\B^m_n(R)$. Any such simplex would therefore have to be contained in the boundary sphere of $\Star_{D^k}(\Delta) = \partial \Delta \ast \Link_{D^k}(\Delta)$. In particular, the two maps $\phi$ and $\phi'$ agree on these simplices. Therefore, $\phi$ and $\phi'$ are	homotopic relative to their boundary spheres. Iterating this procedure, we can homotope $\phi$ to a map $\phi^\dag$ whose image does not contain any externally additive simplices.  The image of such a map has to be contained in $\B^m_n(R)$, and therefore $[\phi] = 0 \in \pi_k(\BX^m_n(R), \B^m_n(R))$.
	\end{proof}
	
	\begin{remark} \label{NotSpherical}
		Let $m \geq 1$. Note that $\BX^m_n(R)$ is \emph{not} $(n-1)$-connected in general: If $n = 2$, then the loop given by the three edges $\{\vec e_{m+1}, \vec e_{m+2}\}, \{\vec e_{m+2}, \vec e_{m+1} + \vec e_{m+2}\}$ and $\{ \vec
		e_{m+1} + \vec e_{m+2}, \vec e_{m+1}\}$ is not nullhomotopic in $\BX^m_2(R)$. This can be seen as follows. Projecting onto $R^2$ gives a splitting of the inclusion $\B_2(R) \m \BX^m_2(R)$. Thus $\pi_1(\B_2(R)) \m \pi_1(\BX^m_2(R))$ is injective. The above loop is clearly nonzero in $\pi_1(\B_2(R))$ since $\B_2(R)$ is a graph, and hence it is nontrivial in $\pi_1(\BX^m_2(R))$ as well.
	\end{remark}
	
	\begin{lemma}\label{lemma:BXvsB}
		Suppose that \autoref{assumption:R} holds, and let $m \geq 1$. If we view $\B_n(R)$ as a subcomplex of $\BX^m_n(R)$ via the standard inclusion 
		$$R^n \hookrightarrow R^{m+n} = R^m \oplus R^n,$$
		then $\pi_k(\BX^m_n(R), \B_n(R))$ and $H_{k}(\BX^m_n(R), \B_n(R))$ are zero for all $k < n$.
	\end{lemma}
	
	\begin{proof}
		If $n = 0$, both $\B_n(R)$ and $\BX^m_n(R)$ are empty spaces, and hence the claim holds. If $n = 1$, then $\B_1(R)$ is a discrete set, and a proper subset of the vertex set of $\BX^m_1(R)$. We need to see that $H_0(\BX^m_1(R), \B_1(R)) = 0$. Let $\vec v = r_1 \vec e_1 + \dots + r_m \vec e_m + r_{m+1} \vec e_{m+1}$ be a vertex in $\BX^m_1(R)$. Then $\{\vec e_1, \dots, \vec e_m, \vec v\}$ is a basis of $R^{m+1}$, and therefore $\vec v_i = r_i \vec e_i + \dots + r_{m+1} \vec e_{m+1}$ is a vertex in $\BX^m_1(R)$ for every $1 \leq i \leq m+1$ because $\{\vec e_1, \dots, \vec e_m, \vec v_i\}$ is also a basis of $R^{m+1}$. Note that $\vec v_1 = \vec v$, and observe that $\vec v_{m+1} = r_{m+1} \vec e_{m+1}$ has to be a vertex in $\B_1(R)$. It follows that every vertex $\vec v \in \BX^m_1(R)$ can be connected to a vertex of $\B_1(R)$ by an edge path consisting of at most $m$ externally additive edges $\{\vec v_i, \vec v_{i+1}\}$ in $\BX^m_1(R)$. Therefore, $\widetilde{H}_0(\B_1(R)) \twoheadrightarrow \widetilde{H}_0(\BX^m_1(R))$ is a surjection, and the claim follows from the long exact sequence in homology.
		
		For $n \geq 2$, we know that $B_n(R)$ is path-connected because it is $(n-2)$-connected and $(n-2) \geq 0$. Furthermore, $\BX^m_n(R)$ is path-connected as argued in \cref{lemma:BXvsBrel}. We will prove that 	$\pi_k(\BX^m_n(R), \B_n(R)) = 0$ for $0 < k < n$. Since $\B_n(R)$ and $\BX^m_n(R)$ are both path-connected, we can then apply the relative Hurewicz theorem to	conclude that $H_k(\BX^m_n(R), \B_n(R)) = 0$ for $k < n$.
		
		Let $0 < k < n$ and consider a simplicial map $\phi\colon (D^{k}, S^{k-1}) \to (\BX^m_n(R), \B_n(R))$, where $D^k$ is a combinatorial disc (see e.g.\ \cite[Definition 6.3 and Lemma 6.4]{putman2009}). By \cref{lemma:BXvsBrel}, we can assume that $\phi\colon (D^{k}, S^{k-1}) \to (\B^m_n(R), \B_n(R))$. Since $\B^m_n(R)$ is Cohen--Macaulay of dimension $(n-1)$, it is locally weakly Cohen--Macaulay of dimension $(n-1)$ in the sense of \cite[§2.1]{galatiusrandalwilliams2018}; that is, the link of a $p$-simplex is isomorphic	to $\B^{m+p+1}_{n-(p+1)}(R)$ and hence $((n-1)-p-2)$-connected. Applying \cite[Theorem 2.4]{galatiusrandalwilliams2018} to $\phi$, we can assume that every vertex $w$ in the interior of $D^k$, i.e.\ $w \not\in S^{k-1}$, satisfies
		$$\phi(\Link_{D^k}(w)) \subseteq \Link_{\B^m_n(R)}(\phi(w)).$$
		
		Let $1 \leq i \leq m$. Pick any vertex $w \in D^k$ with $\phi(w) = \vec z$ and such that the $\vec e_i$-coordinate $r_i$ of $\vec z$ is nonzero. Then $\vec z$ is not contained in $\B_n(R)$, so $w$ has to be contained in the
		interior of $D^k$ since $\phi$ maps the boundary sphere $S^{k-1}$ to $\B_n(R)$. We will explain how one can homotope $\phi$ inside $\BX^m_n(R)$ to a	simplicial map $$\phi'\colon (D^{k}, S^{k-1}) \to (\B^m_n(R), \B_n(R))$$ such that
		\begin{itemize}
			\item the simplicial structure of the domain of $\phi$ and $\phi'$ are exactly the same,
			\item $\phi(v) = \phi'(v)$ for all vertices of $D^k$ except the vertex $w$,
			\item and $\phi'(w) = \vec z - r_i \vec e_i $. Note that the $\vec
			e_i$-coordinate of $\vec z - r_i \vec e_i$ is zero and no other coordinates have
			been changed. 
		\end{itemize}
		We start by observing that $\{\vec z, \vec z - r_i \vec e_i \}$ is an externally additive edge in $\BX^m_n(R)$. Since the image $\phi(D^k) \subset \B^m_n(R)$ of $\phi$ cannot contain externally additive edges, it follows that $\vec z - r_i \vec e_i$ is not contained in $\phi(\Link_{D^k}(w))$. Since $w$ is contained in the interior of $D^k$, the previous paragraph implies that $\phi(\Link_{D^k}(w)) \subseteq \Link_{\B^m_n(R)}(\vec z)$. In particular, $\phi(x) \neq \vec z$ if $x \in \Link_{D^k}(w)$. Hence, the following holds for every simplex $\tau$ of $\Link_{D^k}(w)$.
		\begin{itemize}
			\item The \emph{disjoint} union $\phi(\tau) \sqcup \{ \vec z, \vec z - r_i \vec e_i \}$ is an externally additive simplex in $\BX^m_n(R)$,
			\item $\phi(\tau) \sqcup \{\vec z\} = \phi(\tau \ast w)$,
			\item and $\phi(\tau) \sqcup \{\vec z - r_i \vec e_i \}$ is a simplex of $\B^m_n(R)$.
		\end{itemize}
		Since every simplex in $\Star_{D^k}(w) = \Link_{D^k}(w) \ast w$ is of the form $\tau \ast w$ for some simplex $\tau \in \Link_{D^k}(w)$, this implies that there exists a simplicial map
		$$ h\colon D^k \cup_{\Star_{D^k}(w)}(\Star_{D^k}(w) \ast w') \to \BX^m_n(R),$$
		where $w'$ is a new vertex, $h(w') = \vec z - r_i \vec e_i$ and $h|_{D^k} = \phi$. This is the homotopy between $\phi$ and $\phi'$.
		
		Iterating this construction, we see that $\phi$ is homotopic (relative to the
		boundary sphere $S^k$) to a map $\phi^\dag$ whose image does not contain
		vertices $\vec z$ with nonzero $\vec e_i$-coordinate. Applying this procedure for every $1
		\leq i \leq m$ implies that $\phi$ is homotopic (relative to the boundary sphere
		$S^k$) to a map $\phi^\dag$ satisfying $\phi^\dag(D^k) \subset \B_n(R)$. This
		completes the proof.
	\end{proof}
	
\subsubsection{Surjectivity of the relative apartment class map}
	Now we are ready to prove \autoref{theorem:generators}.
	
	\begin{proof}[Proof of \autoref{theorem:generators}]
		All homology in this proof is taken with $\Z$ coefficients, but we frequently omit the $\Z$ from the notation for readability. Let $0 \leq j \leq n$ and consider a symbol
		$$\Theta = [\vec v_1, \vec v_1+\alpha_1 \vec e_{\beta(1)}]\ast\dots\ast [\vec v_j, \vec v_j+\alpha_j \vec e_{\beta(j)}] \in \relstgen^m_n(j).$$
		Associated with $\Theta$, we find a simplex $\sigma(\Theta) = \{\vec v_1, \dots, \vec v_j\} \in \B^m_n(R)$ and a $\Theta$-dependent spanning map
		$$\Span_{\Theta} \colon \bS^{j-1} \ast \bLink_{\B^m_n(R)}(\sigma(\Theta)) \to \bT_n^m(R).$$
		Passing to reduced homology and applying the suspension isomorphism $j$ times, we obtain a map
		$$f_\Theta \colon \widetilde{H}_{n-1-j}(\bLink_{\B^m_n(R)}(\sigma(\Theta))) \to \widetilde{H}_{n-1}(\bT^m_n(R)).$$
		We will prove by induction on $j$ that the map
		$$F_j = \oplus f_\Theta \colon \bigoplus_{\relstgen^m_n(j)} \widetilde{H}_{n-1-j}(\bLink_{\B^m_n(R)}(\sigma(\Theta))) \twoheadrightarrow \widetilde{H}_{n-1}(\bT^m_n(R))$$
		is a $\GL^m_n(R)$-equivariant surjection for any $0\leq j\leq n$. Here, the right $\GL^m_n(R)$-action on the left-hand-side is described as follows. An element $\phi \in \GL^m_n(R)$ maps a class 
		$$[c] \in \widetilde{H}_{n-1-j}(\bLink_{\B^m_n(R)}(\sigma(\Theta)))$$
		in the summand indexed by $\Theta \in \relstgen^m_n(j)$ to the class 
		$$[c \cdot \phi] \in \widetilde{H}_{n-1-j}(\bLink_{\B^m_n(R)}(\sigma(\Theta \cdot \phi)))$$ in the summand indexed by $(\Theta \cdot \phi) \in \relstgen^m_n(j)$.\\
		
		For $j=n$, we note that $\relstgen^m_n(n) = \relstgen^m_n$ and that $\bLink_{\B^m_n(R)}(\sigma(\Theta)) = \emptyset$ for every $\Theta \in \relstgen^m_n$. In this case, $\widetilde{H}_{-1}(\bLink_{\B^m_n(R)}(\sigma(\Theta)))\cong \Z$ for every $\Theta \in \relstgen^m_n$ and
		$$F = F_n \colon \Z\{\relstgen^m_n\} =  \bigoplus_{\relstgen^m_n} \Z \twoheadrightarrow \widetilde{H}_{n-1}(\bT^m_n) = \St^m_n(R)$$
		is the map in the statement of the theorem. We will now explain the induction argument.\\
		
		For the base case, we consider $j=0$. In this case, $\relstgen^m_n(0)$
		contains a unique element, the empty symbol $\Theta$, and the $\Theta$-dependent spanning map
		$$\Span_{\Theta} \colon \emptyset \ast \bLink_{\B^m_n(R)}(\emptyset) = \B^m_n(R) \to \bT_n^m$$
		agrees with the spanning map studied in \autoref{lemma:BvsT}. Therefore, \autoref{lemma:BvsT} implies that $F_0$ is an $\GL^m_n(R)$-equivariant surjection.\\
		
		Suppose the statement holds for $j \geq 0$. If $j\leq n-1$, we will prove the statement for $j+1$.\\
		
		Consider a symbol $\Theta \in \relstgen^m_n(j)$ and consider any decomposition $R^{m+n}=R^m \oplus \langle \sigma(\Theta) \rangle \oplus C_{\sigma(\Theta)}$. Since $j \neq n$, it holds that $C_{\sigma(\Theta)} \neq 0$. By \cref{assumption:R}, $\widetilde{H}_*(\B(C_{\sigma(\Theta)})) \cong \widetilde{H}_*(\B_{n-j}(R)) = 0$ unless $* =n-j-1$ and $\widetilde{H}_*(\Link_{\B^m_n(R)}(\sigma(\Theta))) \cong \widetilde{H}_*(\B_{n-j}^{m+j}(R)) = 0$ unless $* =n-j-1$. It follows that the long exact sequence for the pair $(\Link_{\B^m_n(R)}(\sigma(\Theta)), \B(C_{\sigma(\Theta)}))$ becomes a short exact sequence
		\begin{multline*}
			0\to \widetilde{H}_{n-1-j}(\B(C_{\sigma(\Theta)}))\to \widetilde{H}_{n-1-j}(\Link_{\B^m_n(R)}(\sigma(\Theta)))\\ \to \widetilde{H}_{n-1-j}(\Link_{\B^m_n(R)}(\sigma(\Theta)),\B(C_{\sigma(\Theta)}))\to 0.
		\end{multline*}
		In particular, the map 
		\[
			\ell_{(\Theta, C_\Theta)} \colon \widetilde{H}_{n-1-j}(\Link_{\B^m_n(R)}(\sigma(\Theta)))\twoheadrightarrow \widetilde{H}_{n-1-j}(\Link_{\B^m_n(R)}(\sigma(\Theta)),\B(C_{\sigma(\Theta)}))
		\]
		is a surjection.\\
		
		By \autoref{lemma:BvsT}, it follows that the following diagram commutes:
		$$
		\begin{tikzcd}
			\widetilde{H}_{n-j-1}(\B(C_{\sigma(\Theta)})) \ar[rr] \ar[d,"\Sigma^j" swap, "\cong"] & & \widetilde{H}_{n-j-1}(\Link_{\B^m_n(R)}(\sigma(\Theta))) \ar[d, "\Sigma^j" swap, "\cong"] \\
			\widetilde{H}_{n-1}(\bS^{j-1} \ast \bB(C_{\sigma(\Theta)})) \ar[rr, "\incl_{(\Theta, C_\Theta)}"] \ar[drr,"0", swap] & & \widetilde{H}_{n-1}(\bS^{j-1} \ast \bLink_{\B^m_n(R)}(\sigma(\Theta))) \ar[d, "\Span_\Theta"] \\
			&& \widetilde{H}_{n-1}(\bT_n^m(R)).
		\end{tikzcd}
		$$
		Notice that the map $f_\Theta$ is the composition of the two rightmost vertical
		maps in the diagram above. We conclude that for any pair $(\Theta, C_\Theta)$
		the map $f_\Theta$ factors over
		$\widetilde{H}_{n-j-1}(\Link_{\B^m_n(R)}(\sigma(\Theta)),\B(C_{\sigma(\Theta)}))$:
		$$
		\begin{tikzcd}
			\widetilde{H}_{n-j-1}(\Link_{\B^m_n(R)}(\sigma(\Theta))) \ar[rr, twoheadrightarrow, "\ell_{(\Theta, C_\Theta)}"] \ar[d, "f_\Theta"] && \widetilde{H}_{n-j-1}(\Link_{\B^m_n(R)}(\sigma(\Theta)),\B(C_{\sigma(\Theta)})) \ar[dll,dashed,  "\Span_{(\Theta, C_\Theta)}"] \\
			\widetilde{H}_{n-1}(T_n^m(R)) & &
		\end{tikzcd}.
		$$
		By \autoref{lemma:BXvsB}, we know that
		\[
			\widetilde{H}_{n-j-1}(\BX^{m, \sigma(\Theta)}_n(R), \B(C_{\sigma(\Theta)}))) \cong \widetilde{H}_{n-j-1}(\BX_{n-j}^{m+j}(R), \B_{n-j}(R))=0.
		\]
		Therefore, the long exact sequence of the triple $(\BX^{m, \sigma(\Theta)}_n(R), \Link_{\B^m_n(R)}(\sigma(\Theta)),\B(C_{\sigma(\Theta)}))$ contains a surjection
		$$\partial_{(\Theta, C_\Theta)}\colon \widetilde{H}_{n-j}(\BX^{m, \sigma(\Theta)}_n(R), \Link_{\B^m_n(R)}(\sigma(\Theta))) \twoheadrightarrow \widetilde{H}_{n-j-1}(\Link_{\B^m_n(R)}(\sigma(\Theta)),\B(C_{\sigma(\Theta)})).$$
		Composing $\partial_{(\Theta, C_\Theta)}$ with the map above, we obtain a map 
		$$g_\Theta = \Span_{(\Theta, C_\Theta)} \circ \partial_{(\Theta,
			C_\Theta)}\colon \widetilde{H}_{n-j}(\BX^{m, \sigma(\Theta)}_n(R), \Link_{\B^m_n(R)}(\sigma(\Theta))) \to \widetilde{H}_{n-1}(T_n^m(R)).$$
		We remark that $g_{\Theta}$ agrees with the map
		$$\widetilde{H}_{n-j}(\BX^{m, \sigma(\Theta)}_n(R),
		\Link_{\B^m_n(R)}(\sigma(\Theta))) \xrightarrow{\partial}
		\widetilde{H}_{n-j-1}(\Link_{\B^m_n(R)}(\sigma(\Theta))) \xrightarrow{f_{\Theta}}
		\widetilde{H}_{n-1}(T_n^m(R)),$$
		since by definition $\partial_{(\Theta, C_\Theta)} = \ell_{(\Theta, C_\Theta)} \circ \partial$ where $\partial$ is the boundary map for the pair $(\BX^{m, \sigma(\Theta)}_n(R),
		\Link_{\B^m_n(R)}(\sigma(\Theta)))$. In particular, $g_{\Theta}$ does not depend on the choice of complement
		$C_{\sigma(\Theta)}$; we only introduced this factorization to check surjectivity in the next step.
		
		If we perform the above construction for every $\Theta \in \relstgen^m_n(j)$, we
		arrive at a map
		\begin{equation*}
			G_{j+1} = \oplus g_\Theta\colon \bigoplus_{\relstgen^m_n(j)} \widetilde{H}_{n-j}(\BX^{m, \sigma(\Theta)}_n(R), \Link_{\B^m_n(R)}(\sigma(\Theta))) \to \widetilde{H}_{n-1}(T_n^m(R)).
		\end{equation*}
		We will show that $G_{j+1}$ is surjective. Then we will relate this map to $F_{j+1}$ and use its surjectivity to show that $F_{j+1}$ is surjective as well.
		
		Let us first verify that $G_{j+1}$ is surjective. By the definition of $G_{j+1}$, there is a commutative diagram
		$$
		\begin{tikzcd}
			\bigoplus\limits_{\relstgen^m_n(j)} \widetilde{H}_{n-j-1}(\Link_{\B^m_n(R)}(\sigma(\Theta))) \ar[rr, twoheadrightarrow, "\oplus \ell_{(\Theta, C_\Theta)}"] \ar[d, "F_j = \oplus f_\Theta"] && \bigoplus\limits_{\relstgen^m_n(j)} \widetilde{H}_{n-j-1}(\Link_{\B^m_n(R)}(\sigma(\Theta)),\B(C_{\sigma(\Theta)})) \ar[dll,dashed,  "\oplus \Span_{(\Theta, C_\Theta)}"] \\
			\widetilde{H}_{n-1}(T_n^m(R)) & &
		\end{tikzcd}
		$$
		By the induction hypothesis, the leftmost arrow in this diagram, $F_j$, is surjective. Therefore, the map
		$$\oplus \Span_{(\Theta, C_\Theta)} \colon \bigoplus_{\relstgen^m_n(j)} \widetilde{H}_{n-j-1}(\Link_{\B^m_n(R)}(\sigma(\Theta)),\B(C_{\sigma(\Theta)})) \twoheadrightarrow \widetilde{H}_{n-1}(T_n^m(R))$$
		is surjective. On the other hand,
		\begin{multline*}
			\oplus \partial_{(\Theta, C_\Theta)} \colon \bigoplus_{\relstgen^m_n(j)} \widetilde{H}_{n-j}(\BX^{m, \sigma(\Theta)}_n(R), \Link_{\B^m_n(R)}(\sigma(\Theta))) \twoheadrightarrow\\ \bigoplus_{\relstgen^m_n(j)} \widetilde{H}_{n-1-j}(\Link_{\B^m_n(R)}(\sigma(\Theta)),\B(C_{\sigma(\Theta)}))
		\end{multline*}
		is surjective as well because each of the functions $\partial_{(\Theta, C_\Theta)}$ is surjective. It follows that
		$$ G_{j+1} = \oplus g_\Theta = \oplus (\Span_{(\Theta, C_\Theta)} \circ \partial_{(\Theta, C_\Theta)}) = (\oplus \Span_{(\Theta, C_\Theta)}) \circ (\oplus \partial_{(\Theta, C_\Theta)})$$
		is surjective as a composition of two surjections.
		
		We will now explain how $G_{j+1}$ can be related to $F_{j+1}$. For this we will decompose the domain of $G_{j+1}$. The definition of $\BX^{m, \sigma(\Theta)}_n(R)$ implies that
		\begin{align*}
			&\widetilde{H}_{n-j}(\BX^{m, \sigma(\Theta)}_n(R), \Link_{\B^m_n(R)}(\sigma(\Theta)))\\
			&= \bigoplus_{\substack{\tau = \{\vec v, \vec v + r \vec e \},\\ \vec v \in \Link_{\B^m_n(R)}(\sigma(\Theta)) \text{ and } r\not=0,\\ \vec e \in \{\vec e_1,\dots, \vec e_{m}\} \cup \sigma(\Theta)}} \widetilde{H}_{n-j}(\Star_{\BX^{m, \sigma(\Theta)}_n(R)}(\tau), \Star_{\BX^{m, \sigma(\Theta)}_n(R)}(\tau) \cap \Link_{\B^m_n(R)}(\sigma(\Theta))).
		\end{align*}
		Furthermore, we have the following four identifications
		\begin{align}
			\nonumber &\widetilde{H}_{n-j}(\Star_{\BX^{m, \sigma(\Theta)}_n(R)}(\tau), \Star_{\BX^{m, \sigma(\Theta)}_n(R)}(\tau) \cap \Link_{\B^m_n(R)}(\sigma(\Theta)))\\
			\label{identification-1} &\cong \widetilde{H}_{n-j-1}(\Star_{\BX^{m, \sigma(\Theta)}_n(R)}(\tau) \cap \Link_{\B^m_n(R)}(\sigma(\Theta))) \\
			\label{identification-2} &\cong \widetilde{H}_{n-j-1}(\partial(\tau) \ast \Link_{\BX^{m, \sigma(\Theta)}_n(R)}(\tau))\\
			\label{identification-3} &\cong \widetilde{H}_{n-j-2}(\Link_{\BX^{m, \sigma(\Theta)}_n(R)}(\tau))\\
			\label{identification-4} &\cong \widetilde{H}_{n-j-2}(\Link_{\B_{n}^{m}(R)}(\sigma(\Theta) \ast \{\vec v\})).
		\end{align}
		Therefore, it follows that
		\begin{align*}
			&\widetilde{H}_{n-j}(\BX^{m, \sigma(\Theta)}_n(R), \Link_{\B^m_n(R)}(\sigma(\Theta)))\\
			&\cong \bigoplus_{\substack{\tau = \{\vec v, \vec v + r \vec e \},\\ \vec v \in \Link_{\B^m_n(R)}(\sigma(\Theta)) \text{ and } r\not=0,\\ \vec e \in \{\vec e_1,\dots, \vec e_{m}\} \cup \sigma(\Theta)}} \widetilde{H}_{n-j-2}(\Link_{\B_{n}^{m}(R)}(\sigma(\Theta) \ast \{ \vec v\}))
		\end{align*}
		if $n-j-1 \geq 0$. Hence, $G_{j+1}$ can be identified with a surjection
		$$\bigoplus_{\relstgen^m_n(j)} \bigoplus_{\substack{\tau = \{\vec v, \vec v + r \vec e \},\\ \vec v \in \Link_{\B^m_n(R)}(\sigma(\Theta)) \text{ and } r\not=0,\\ \vec e \in \{\vec e_1,\dots, \vec e_{m}\} \cup \sigma(\Theta)}} \widetilde{H}_{n-j-2}(\Link_{\B_{n}^{m}(R)}(\sigma(\Theta) \ast \{ \vec v\})) \twoheadrightarrow \widetilde{H}_{n-1}(T_n^m(R)).$$
		
		Consider the following symbol $\Theta^\dag \in \relstgen^m_n(j+1)$,
		$$\Theta^\dag = [ \vec v_1, \vec v_1+ r_1 \vec e_{\beta(1)}]\ast\dots \ast [\vec v_j, \vec v_j+ r_j \vec e_{\beta(j)}] \ast [\vec v_{j+1}, \vec v_{j+1}+ r_{j+1} \vec e_{\beta(j+1)}].$$
		Forgetting the last pair $[\vec v_{j+1}, \vec v_{j+1}+ r_{j+1} \vec e_{\beta(j+1)}]$ in $\Theta^\dag$, we obtain a symbol
		$$\Theta = [\vec v_1, \vec v_1+ r_1 \vec e_{\beta(1)}]\ast\dots \ast [\vec v_j, \vec v_j+ r_j \vec e_{\beta(j)} ] \in \relstgen^m_n(j).$$
		Forgetting the order of the pair $[\vec v_{j+1}, \vec v_{j+1}+ r_{j+1} \vec e_{\beta(j+1)}]$, we obtain an externally additive edge $\tau = \{\vec v_{j+1}, \vec v_{j+1}+ r_{j+1} \vec e_{\beta(j+1)} \}$ in $\BX^{m, \sigma(\Theta)}_n(R)$. Noting that $\sigma(\Theta^\dag) = \sigma(\Theta) \ast \{\vec v_{j+1}\}$, we obtain a forgetful surjection
		\begin{multline*}
			U_{j+1}\colon \bigoplus_{\relstgen^m_n(j+1)} \widetilde{H}_{n-j-2}(\Link_{\B_{n}^{m}(R)}(\sigma(\Theta^\dag))) \twoheadrightarrow\\ 
			\bigoplus_{\relstgen^m_n(j)} \bigoplus_{\substack{\tau = \{\vec v, \vec v + r \vec e \},\\ \vec v \in \Link_{\B^m_n(R)}(\sigma(\Theta)) \text{ and } r\not=0,\\ \vec e \in \{\vec e_1,\dots, \vec e_{m}\} \cup \sigma(\Theta)}} \widetilde{H}_{n-j-2}(\Link_{\B_{n}^{m}(R)}(\sigma(\Theta) \ast \{ \vec v\})).
		\end{multline*}
		
		We claim that
		$$ G_{j+1} \circ U_{j+1} = F_{j+1}.$$
		Once this is checked, the induction step and the proof of the theorem are complete.
		
		To see that this equality holds, we consider the summand $\widetilde{H}_{n-j-2}(\Link_{\B_{n}^{m}(R)}(\sigma(\Theta^\dag)))$ indexed by $\Theta^\dag \in \relstgen^m_n(j+1)$ of the domain of $F_{j+1}$ and verify that the restriction of $G_{j + 1} \circ U_{j + 1}$ to it is equal to the function $f_{\Theta^\dag}$. Since $F_{j+1} = \oplus_{\Theta^\dag \in \relstgen^m_n(j+1)} f_{\Theta^\dag}$, the claim then follows.
		
		Applying the definition of $U_{j+1}$, we forget the last pair in the symbol $\Theta^\dag$ to obtain the symbol $\Theta\in \relstgen^m_n(j)$ and the externally additive edge $\tau = \{\vec v_{j+1}, \vec v_{j+1} + r_{j+1} \vec e_{\beta(j+1)}\}$ in $\BX^{m, \sigma(\Theta)}_n(R)$. The map $f_{\Theta^\dag}$ was obtained from the $\Theta^{\dag}$-dependent spanning map
		$$\Span_{\Theta^\dag} \colon \bS^{j} \ast \bLink_{\B^m_n(R)}(\sigma(\Theta^\dag)) \to \bT_n^m(R)$$
		and $(j+1)$-many suspension isomorphisms. The last one of the $0$-spheres in the join $\bS^j = \bS^0 \ast \dots \ast \bS^0$ is mapped to the pair $\{\langle \vec v_{j+1} \rangle, \langle \vec v_{j+1} + r_{j+1} \vec e_{\beta(j+1)} \rangle\}$ in $\bT_n^m(R)$. Hence, we can relate $f_{\Theta^\dag}$ to $f_{\Theta}$ by writing $f_{\Theta^\dag}$ as the following composition:
		$$
		\begin{tikzcd}
			\widetilde{H}_{n-j-2}(\bLink_{\B_n^m(R)}(\sigma(\Theta^{\dag}))) 
			\ar[d, "\Sigma_{\partial \tau}", "\cong"']
			\ar[ddr, "f_{\Theta^\dag}"', bend left]\\
			\widetilde{H}_{n-j-1}(\partial(\tau) \ast \bLink_{\B_n^m(R)}(\sigma(\Theta^{\dag}))) 
			\ar[d, "\text{inclusion}"]
			& \\
			\widetilde{H}_{n-j-1}(\bLink_{\B_n^m(R)}(\sigma(\Theta))) 
			\ar[r, "f_{\Theta}"]
			& \widetilde{H}_{n-1}(T_n^m(R))
		\end{tikzcd}
		$$
		Using the identifications in \cref{identification-1,identification-2,identification-3,identification-4}, that $\sigma(\Theta^\dag) = \sigma(\Theta) \cup \{\vec v_{j+1}\}$ and abbreviating $\Star_{\BX^{m, \sigma(\Theta)}_n(R)}(\tau)$ as $\Star(\tau)$,  we also find at the following commutative diagram which contains the above factorization of $f_{\Theta^\dag}$ as the left vertical composition.
		$$
		\begin{tikzcd}
			\widetilde{H}_{n-j-2}(\Link_{\B^m_n(R)}(\sigma(\Theta^\dag))) 
			\ar[r, "U_{j+1}", equal]
			\ar[d, "\Sigma_{\partial \tau}", "\cong"']
			& \widetilde{H}_{n-j-2}(\Link_{\B^{m}_n(R)}(\sigma(\Theta) \cup \{\vec v_{j+1}\})
			\ar[d, "\cong"', "\text{\cref{identification-4,identification-3}}"]\\
			\widetilde{H}_{n-j-1}(\partial(\tau) \ast \Link_{\B^m_n(R)}(\sigma(\Theta^\dag))
			\ar[r, "\cong", "\text{\cref{identification-4}}"' {yshift=-7.5pt}]
			\ar[d, "\text{inclusion}"]
			& \widetilde{H}_{n-j-1}(\partial(\tau) \ast \Link_{\BX^{m, \sigma(\Theta)}_n(R)}(\tau))
			\ar[d, "{\text{inclusion, }\cong}"', "\text{\cref{identification-2}}"]
			\\
			\widetilde{H}_{n-j-1}(\Link_{\B_n^m}(\sigma(\Theta)))
			\ar[dddr, "\ell_{(\Theta, C_{\Theta})}"' near end, bend right=25]
			\ar[ddddr, "f_\Theta"', bend right=40]
			& \widetilde{H}_{n-j-1}(\Star(\tau)\cap \Link_{\B_n^m}(\sigma(\Theta))) 
			\ar[l, "\text{inclusion}" {yshift=-3pt}]
			\ar[d, "{\text{connecting map, } \cong}"', "\text{\cref{identification-1}}"]\\
			& \widetilde{H}_{n-j}(\Star(\tau), \Star(\tau)\cap \Link_{\B_n^m}(\sigma(\Theta))) \ar[d, "\text{inclusion}"]\\
			& \widetilde{H}_{n-j}(\BX^{m, \sigma(\Theta)}_n(R), \Link_{\B_n^m(R)}(\sigma(\Theta)))
			\ar[d, "{\partial_{(\Theta, C_\Theta)}}"]
			\ar[luu, "\text{connecting map}"] \ar[dd, "g_\Theta", bend left=80] \\
			& \widetilde{H}_{n-j-1}(\Link_{\B_n^m}(\sigma(\Theta)), \B(C_{\sigma(\Theta)})) \ar[d, "{\on{span}_{(\Theta, C_\Theta)}}"]\\
			& \widetilde{H}_{n-1}(T_n^m(R))
		\end{tikzcd}
		$$
		Note that on the right vertical composition of this diagram is exactly the definition of $G_{j+1} \circ U_{j+1}$ on the summand $\widetilde{H}_{n-j-2}(\Link_{\B^m_n(R)}(\sigma(\Theta^\dag)))$. Hence the claim follows.
	\end{proof}
	\section{Slope-1 homological stability}
	\label{sec:slope-1-homological-stability}
In this section, we describe how the slope-1 homological stability result, \autoref{theoremA}, follows from \cref{corollary:coinvariants-of-steinberg} and \cref{corollary:coinvariants-of-relative-steinberg}. This section of the paper is very similar to parts of \cite{galatiuskupersrandalwilliams2018cellsandfinite}, relies on the $E_k$-cellular approach to homological stability develop in \cite{galatiuskupersrandalwilliams2021cells} and is almost identical to parts of \cite{kupersmillerpatzt2022}. Because we do not assume that $R$ is a principle ideal domain as in the setting of \cite{kupersmillerpatzt2022}, we include the details of technical arguments in the \cref{sec:appendix} to show how their work carries over to our setting (i.e.\ using \cref{assumption:R} only).  We refer the reader to \cite{kupersmiller2018,galatiuskupersrandalwilliams2021cells,galatiuskupersrandalwilliams2018cellsandfinite,galatiuskupersrandalwilliams2019cellsandmcg,galatiuskupersrandalwilliams2020cellsandinfinite,kupersmillerpatzt2022} for the required background on $E_k$-cells.

\subsection{$E_1$-homology and the Charney module}
\label{subsec:e1-homology-and-charneys-splitting-complex}

Suppose that $R$ satisfies \cref{assumption:R}. Let $\GL(R)$ denote the groupoid with set of objects given by the natural numbers $\N$ and automorphisms of $n$ given by $\GL_n(R)$. The block sum operation $\oplus$ and the swap automorphisms $R^m \oplus R^n \to R^n \oplus R^m$ yield a symmetric monoidal structure $(\GL(R), \oplus, 0)$, with the property that $r: \GL(R) \to \N: n \mapsto n$ is a symmetric monoidal functor and $r^{-1}(0) = 0$. We are hence exactly in the setting of \cite[Section 17.1]{galatiuskupersrandalwilliams2021cells}. Furthermore, $\GL_0(R)$ is the trivial group and the block sum map $- \oplus - : \GL_m(R) \times \GL_n(R) \to \GL_{m+n}(R)$ is injective, so \cite[Assumption 17.1 and 17.2]{galatiuskupersrandalwilliams2021cells} hold. It follows from \cite[Corollary 17.5, Lemma 17.10 and Remark 17.11]{galatiuskupersrandalwilliams2021cells} and \cref{lemma:splitting-complex-vs-splitting-poset} that the $E_1$-homology of the $E_\infty$-algebra $\on{BGL}(R) = \bigsqcup_{n \in \N} \on{BGL}_n(R)$ can be interpreted in terms of the equivariant homology of the following version of a splitting complex first studied by Charney \cite{charney1980}.

\begin{definition}
	\label{definition:e1-splitting-complex}
	Let $\bS_n^{E_1}(R)$ denote the poset whose elements are pairs of nonzero free submodules $(P,Q)$ with $P \oplus Q=R^n$, where $(P,Q) \leq (P',Q')$ if $P \subseteq P'$ and $Q' \subseteq Q$. Let $S_n^{E_1}(R)$ denote the geometric realization of $\bS_n^{E_1}(R)$.
\end{definition}

Charney considered a slight variant of this complex for general Dedekind domains, where the assumption that $P$ and $Q$ are free is not included, and proved that these complexes are spherical \cite[Theorem 1.1]{charney1980}. In this case, their equivariant homology measures the $E_1$-homology of an $E_\infty$-algebra built out of automorphisms of (not necessarily free) projective modules.

Note that $\bS_n^{E_1}(R)$ is $(n-2)$-dimensional if \cref{assumption:R} holds, and that for principal ideal domains $R$ the complex in \cref{definition:e1-splitting-complex} agrees with Charney's \cite{charney1980}. This leads us to the following definition.

\begin{definition}
	\label{def:charney-module}
	Let $R$ be a ring such that $\bS_n^{E_1}(R)$ is $(n-2)$-spherical. Then the \emph{Charney module} of $R^n$ is the right $\GL_n(R)$-module $\on{Ch}_n(R):= \widetilde H_{n-2}(S_n^{E_1}(R); \Z)$.
\end{definition}

\begin{remark}
	\label{remark:charney-module-is-e1-steinberg}
	The  Charney module in \cref{def:charney-module} is an example of an \emph{$E_1$-Steinberg module} as introduced in \cite[Definition 17.6]{galatiuskupersrandalwilliams2021cells}, compare with \cref{lemma:splitting-complex-vs-splitting-poset}.  For this reason, it is frequently referred to as the $E_1$-Steinberg module or split Steinberg module of the $E_\infty$-algebra $\on{BGL}(R)$ in the literature \cite{galatiuskupersrandalwilliams2021cells, galatiuskupersrandalwilliams2018cellsandfinite, galatiuskupersrandalwilliams2020cellsandinfinite, kupersmillerpatzt2022}.
\end{remark} 

 The first part of the next proposition extends Charney's connectivity theorem \cite[Theorem 1.1]{charney1980} to our setting. The second part is a vanishing result for the coinvariants of the Charney modules and is obtained using \cref{corollary:coinvariants-of-steinberg} and \cref{corollary:coinvariants-of-relative-steinberg}. These two results constitute the computational input required by the $E_k$-cellular approach to homological stability to prove our slope-1 homological stability theorem.

\begin{proposition}
	\label{proposition:computational-input-split-steinberg}
	If $R$ satisfies \autoref{assumption:R}, then $\bS_n^{E_1}(R)$ is $(n-2)$-spherical for all $n \geq 0$. Moreover, if $n \geq 2$, then $(\on{Ch}_n(R) \otimes \K)_{\GL_n(R)} \cong 0$ for every commutative ring $\K$ in which $2 \in \K^\times$ is a unit.
\end{proposition}

\begin{proof}
	\cite[Theorem 4.8]{kupersmillerpatzt2022} proves this claim if $R$ is Euclidean and under different connectivity assumptions. However, their argument remains valid if $R$ satisfies \cref{assumption:R}. This uses the fact that $R^{op}$ also satisfies \cref{assumption:R} by the virtue of \cref{theorem:relation-to-Rop}, as well as the results established in \cref{lemma:span-map}, \cref{corollary:coinvariants-of-steinberg} and \cref{corollary:coinvariants-of-relative-steinberg}. To make this article self-contained, we explain how one can adapt their argument to our setting in \cref{appendix:standard-connectivity-estimate-and-coinvariants-of-the-charney-module}.
\end{proof}

\subsection{Proof of the stability theorem}
We are now ready to prove \cref{theoremA}.

\begin{proof}[Proof of \cref{theoremA}]
	Let $\N_{> 0}$ denote the positive integers. The framework of \cite[Section 17.1]{galatiuskupersrandalwilliams2021cells} and the discussion in \cite[Section 18.1]{galatiuskupersrandalwilliams2021cells} gives a functorially defined $\N_{> 0}$-graded non-unital $E_\infty$-algebra $\mathbf{R}_{\K}$ in the category of simplicial $\K$-modules with $H_{n,i}(\mathbf{R}_{\K})=H_{i}(\GL_n(R);\K)$. Here $H_{n,i}$ means homology in homological	degree $i$ and grading $n$. Using that $\bS_n^{E_1}(R)$ is $(n-2)$-spherical by the first part of \cref{proposition:computational-input-split-steinberg}, it follows from \cite[Lemma	18.2, Corollary 17.5, Lemma 17.10, and Remark 17.11]{galatiuskupersrandalwilliams2021cells} and \cref{lemma:splitting-complex-vs-splitting-poset} that the derived $E_1$-indecomposables of $\mathbf{R}_{\K}$ can be identified with the homology of the general linear groups with coefficients in the Charney modules (see also \cite[Proposition 2.3]{kupersmillerpatzt2022}),
	\[
		H_{n,i}^{E_1}(\mathbf{R}_{\K}) \cong H_{i-n+1}(\GL_n(R);\on{Ch}_n(R) \otimes \K).
	\]
	In particular, $H_{n,i}^{E_1}(\mathbf{R}_{\K}) \cong 0$ for $i \leq n-2$ for degree reasons. Invoking the second part of \cref{proposition:computational-input-split-steinberg} and using the assumption that $2 \in \K^\times$ is a unit, it follows that $H_0(\GL_n(R);\on{Ch}_n(R) \otimes \K) \cong 0$ if $n \geq 2$. Thus, $H_{n,n-1}^{E_1}(\mathbf{R}_{\K}) \cong 0$ for $n \geq 2$. Therefore \cite[Proposition 5.1]{kupersmillerpatzt2022} implies \cref{theoremA}: the conclusion of \cite[Proposition 5.1]{kupersmillerpatzt2022} concerns the vanishing of the homology of a simplicial $\K$-module $\overline{\mathbf{R}}_{\K}/\sigma$ and, by definition, $H_{n,i}(\overline{\mathbf{R}}_{\K}/\sigma) =	H_i(\GL_n(R), \GL_{n-1}(R);\K)$. The vanishing of these relative homology groups is equivalent to the homological stability result claimed in \cref{theoremA}.
\end{proof}
	
	\appendix
	\section{Additional details}
	\label{sec:appendix}

\subsection{(Relative) Tits complexes}

In this section we verify the identifications used in arguments involving the Tits complexes associated to $R$, \cref{def:tits-building}. These are simple adaptations of well-known results to our setting.

\begin{lemma}
	\label{upperlowerTits}
	Let $R$ satisfy \cref{assumption:R} and $V, V' \in \bT_n(R)$ such that $V \subsetneq V'$. Then
	\begin{enumerate}
		\item \label{item-upper-fiber-tits-complex} $\bT_n(R)_{> V} \cong \bT(R^n/V) \cong \bT_{n - \rank(V)}(R)$;
		\item \label{item-lower-fiber-tits-complex} $\bT_n(R)_{< V} = \bT(V) \cong \bT_{\rank(V)}(R)$;
		\item \label{item-interval-tits-complex} $\bT_n(R)_{(V, V')} = \bT(V')_{> V} \cong \bT_{\rank(V') - \rank(V)}(R)$.
	\end{enumerate}
\end{lemma}

\begin{proof}
	For \cref{item-upper-fiber-tits-complex}, we will check that $\bT_n(R)_{> V} \cong \bT(R^n/V).$ Then \cref{lemma:complements-are-free} implies that $R^n/V$ is free of rank $(n - \rank(V))$ and hence it follows that $\bT(R^n/V) \cong \bT_{n - \rank(V)}(R)$. We claim that
	$$\bT_n(R)_{> V} \to \bT(R^n/V): U \mapsto U/V \text{ and } \bT(R^n/V) \to \bT_n(R)_{> V}: W \mapsto q_V^{-1}(W)$$
	are inverse poset maps, where $q_V: R^n \to R^n/V$ denotes the projection map.
	The first map is well-defined: If $R^n = U \oplus C'$, then $R^n/V \cong  (C'\oplus U)/V \cong C'\oplus U/V$, and it follows from \cref{lemma:complements-are-free} that $U/V$ is a free summand of $R^n/V$.  The first map is order-preserving: If $U \subseteq U'$, then $U/V \subseteq U'/V$. The second map is well-defined: If $R^n/V = W \oplus D$, then \cref{lemma:complements-are-free} implies that $0 \to q_{V}^{-1}(W) \to R^n \to D \to 0$ is split and that $q_V^{-1}(W)$ is a free summand of $R^n$. The second map is order-preserving: If $W \subseteq W'$, then $q_{V}^{-1}(W) \subseteq q_{V}^{-1}(W')$. Its  easy to check that these two maps are inverses of each other. The first equality in \cref{item-lower-fiber-tits-complex} follows from \cref{lemma:summand-property}, the second identification is immediate. \cref{item-interval-tits-complex} follows from \cref{item-upper-fiber-tits-complex} and \cref{item-lower-fiber-tits-complex}.
\end{proof}

\begin{lemma}
	\label{upperlowerRelTits}
	Let $R$ satisfy \cref{assumption:R}, $m > 0$ and $V \in \bT^m_n(R)$. Then
	\begin{enumerate}
		\item \label{item-upper-fiber-relative-tits-complex} $\bT^m_n(R)_{> V} \cong \bT(R^{n+m}/V, R^m) \cong \bT^m_{n - \rank(V)}(R)$;
		\item \label{item-lower-fiber-relative-tits-complex} $\bT^m_n(R)_{< V} = \bT(V) \cong \bT_{\rank(V)}(R)$.
	\end{enumerate}
\end{lemma}

\begin{proof}
	The first identification in \cref{item-upper-fiber-relative-tits-complex}, $\bT^m_n(R)_{> V} \cong \bT(R^{n+m}/V , R^m)$, is obtained by restriction from the poset isomorphism $\bT_{m+n}(R)_{> V} \cong \bT(R^{m+n}/V)$ constructed for \cref{item-upper-fiber-tits-complex} of \cref{upperlowerTits}. For the second isomorphism we consider a complement $L$ of $R^m$ in $R^{m+n}$ such that $V \subseteq L$. Then $L$ is free by \cref{lemma:complements-are-free}, $L = V \oplus C$ by \cref{lemma:summand-property} and we obtain a decomposition $V \oplus C \oplus R^m = R^{m+n}$, where $C$ is free of rank $(n - \rank(V))$ by \cref{lemma:complements-are-free}. Hence, $R^{m+n} /V = (V \oplus C \oplus R^m) / V \cong V/V \oplus C \oplus R^m \cong C \oplus R^m $ is free of rank $(n - \rank(V) + m)$ and therefore $\bT(R^{n+m}/V , R^m) \cong \bT^m_{n - \rank(V)}(R)$. The first equality in \cref{item-lower-fiber-tits-complex} follows from \cref{lemma:summand-property}, and the second identification is immediate.
		
	For \cref{item-lower-fiber-relative-tits-complex}, the equality $\bT^m_n(R)_{< V} = \bT(V)$, and hence the claim, follows by restriction from the equality $\bT_{m+n}(R)_{< V} = \bT(V)$ proved in \cref{item-lower-fiber-tits-complex} of \cref{upperlowerTits}.
\end{proof}

\subsection{The complex of partial bases of the opposite ring}
\label{appendix:partial-bases-of-the-opposite-ring}

In this part of the appendix, we complete the proof of \cref{theorem:relation-to-Rop} claiming that $R^{op}$ satisfies \cref{assumption:R} if $R$ does. For the first two items of \cref{assumption:R} this is contained in the literature, as explained in \cref{lemma:relation-to-Rop}. Our contribution is to show that $R^{op}$ also satisfies \cref{assumption-partial-bases} of \cref{assumption:R}, i.e.\ that $\B_n(R^{op})$ is Cohen--Macaulay. To see this, we will adapt an argument due to Sadofschi Costa \cite{sadofschicosta2020} to our setting.

We start with an several preliminary observations; the first one shows how the Tits complexes and Steinberg modules of $R$ and $R^{op}$ are related.

\begin{lemma}
	\label{lemma:comparing-the-tits-complexes-of-R-and-Rop}
	Let $R$ be a ring satisfying \cref{assumption:R} and let $M$ be a free $R$-module of rank $n$. Let $M^\vee = \Hom_R(M, R)$ be its dual $R^{op}$-module, which is free of rank $n$. There is an isomorphism of posets of Tits complexes
	$$\bT(M) \to \bT(M^\vee)^{op}: V \mapsto V^\circ$$
	mapping a summand $V$ to the summand $V^\circ = \{f \in M^\vee : f|_V = 0\}$. This isomorphism is compatible with the action of $\GL(M)$ and $\GL(M^\vee)$ on the domain and codomain if these two groups are identified using the inverse-transpose isomorphism in \cref{lemma:dualizing-and-inverse-transpose}. 
\end{lemma}

\begin{proof}
	We first check that the map $\bT(M) \to \bT(M^\vee)^{op}: V \mapsto V^\circ$ is well-defined: Let $V \in \bT(M)$. Then restricting $R$-linear functions to the summand $V$ yields a surjection $M^\vee \twoheadrightarrow V^\vee$, whose kernel is $V^\circ$. Since $V^\vee$ is a free $R^{op}$-module of rank $\rank(V)$ and since $R^{op}$ is Hermite by \cref{lemma:relation-to-Rop}, it follows that $V^\circ$ is a free summand of rank $(n-\rank(V))$. If $V_1 \subseteq V_2$, then it obviously holds $V_2^\circ \subseteq V_1^\circ$. The map is hence a well-defined poset map.
	
	The canonical isomorphism $M \cong (M^\vee)^\vee$ (see e.g.\ \cite[Theorem 9.2]{blyth1990}) induces a poset isomorphism $\bT(M) \to \bT((M^\vee)^\vee)$ identifying $V$ and $(V^\circ)^\circ$. Since this isomorphism factors through $\bT(M^\vee)^{op}$, the poset map $V \mapsto V^\circ$ is injective. Using the canonical isomorphism $M^\vee \cong ((M^\vee)^\vee)^\vee$ one can similarly check that the poset map is surjective. To conclude that the map is an isomorphism of posets, we are left with checking that $V_1 \subseteq V_2$ if $V_2^\circ \subseteq V_1^\circ$: Let $\vec v \in V_1$. Then $f(\vec v) = 0$ for all $f \in V_2^\circ \subseteq V_1^\circ$. Hence $\on{ev}_{\vec v}(f) = 0$ for all $f \in V_2^\circ$, which means that $\on{ev}_{\vec v} \in (V_2^\circ)^\circ$. Now the canonical isomorphism implies that $\vec v \in V_2$. It follows that the map is a poset isomorphism as claimed.
	
	Finally, we check the compatibility with the group actions, i.e.\ $(V \cdot \phi)^\circ = (V^\circ)\cdot( \phi^{-1})^*$ for $V \in \bT(M)$ and $\phi \in \GL(M)$. To see this, we note that $f_0 \in (V^\circ)\cdot( \phi^{-1})^*$ if and only if $f_0 = f_1 \circ \phi^{-1}$ for $f_1 \in V^\circ$; if and only if $f_0 \circ \phi \in V^\circ$; if and only if $f_0(\phi(\vec v)) = 0$ for $\vec v \in V$; if and only if $f_0 \in (V \cdot \phi)^\circ$.
\end{proof}

\begin{corollary}
	Let $R$ be a ring satisfying \cref{assumption:R}, let $M$ be a finite rank $R$-module and consider its dual $R^{op}$-module $M^\vee$. Then $T(M^\vee)$ is Cohen--Macaulay of dimension $(n-2)$ and the apartment class map $[-]\colon \Z[\GL(M^\vee)] \to \St(M^\vee)$ is an equivariant surjection.
\end{corollary}

\begin{proof}
	The claim follows from \cref{lemma:dualizing-and-inverse-transpose}, \cref{lemma:span-map}, \cref{theorem:apartment-classes-generate-steinberg}, \cref{lemma:comparing-the-tits-complexes-of-R-and-Rop} and the resulting commutative diagram
	\begin{center}
		\begin{tikzcd}
			\Z[\GL(M)] \arrow[rr, "{((-)^{-1})^*}", "\cong"'] \arrow[d, two heads, "{[-]}" near start] && \Z[\GL(M^\vee)] \arrow[d, "{[-]}" near start]\\
			\St(M) \arrow[rr, "{(-)^{\circ}_*}", "\cong"'] && \St(M^\vee)
		\end{tikzcd}
	\end{center}
\end{proof}

We now start working towards establishing a relation between the complexes of partial bases of $R$ and $R^{op}$. \cref{lemma:comparing-the-tits-complexes-of-R-and-Rop} shows that we can compare summands of a free $R$-module $M$ to summands of the $R^{op}$-module $M^\vee$, even though it is not clear how unimodular vectors in the former can be related to unimodular vectors in the latter. The idea is therefore to pass to certain simplicial complexes, which are built out of summands but ``remember'' the homotopy type of the complex of partial bases. This is the role of the following two complexes.

\begin{definition}
	\label{definition:frame-and-coframe-complexes}
	Let $M$ be a free $R$-module of finite rank $n$. A \emph{frame} in $M$ is a set $\{L_1, \dots, L_n\}$ of rank-1 summands in $M$ such that there exists a basis $\{\vec v_1, \dots, \vec v_n\}$ of $M$ with $L_i = \langle \vec v_i \rangle$. We denote by $\on{F}(M)$ the complex of partial frames, whose $k$-simplices are size-$(k+1)$ subsets of frames. We write $\on{F}_n(R)$ for $\on{F}(R^n)$. A \emph{co-frame} in $M$ is a set $\{C_1, \dots, C_n\}$ of rank-$(n-1)$ summands in $M$ with the property that there exists a basis $\{v_1, \dots, v_n\}$ of $M$ such $C_i = \langle \vec v_j : j \neq i \rangle$. We denote by $\on{coF}(M)$ the complex of partial co-frames, whose $k$-simplices are size-$(k+1)$ subsets of co-frames. We write $\on{coF}_n(R) = \on{coF}(R^n)$.
\end{definition}

Complexes of frames related to those in \cref{definition:frame-and-coframe-complexes} have previously been considered, see e.g.\ \cite{churchfarbputman2019} and \cite{kupersmillerpatzt2022}. The complex of co-frames $\on{coF}_n(R)$ can be seen as an analogue of the simplicial complex $Y_H$ defined and studied by Hatcher--Vogtmann in the context of free groups, see \cite[Remark after Corollary 3.4 and Theorem 2.4]{hatchervogtmann1998}. Hatcher--Vogtmann's complex plays a key role in Sadofschi Costa's work \cite{sadofschicosta2020}.

In the next few lemmas, we make precise in which sense these complexes ``remember'' the homotopy type of the complex of partial bases. An important tool for this is Hatcher--Wahl's notion of a complete join complex, see \cite[Definition 3.2 and Example 3.3]{hatcherwahl2010}.

\begin{lemma}
	\label{lemma:complete-join-complex}
	Let $Y$ be a complete join complex over a simplicial complex $X$ of dimension $d$. Then $Y$ is Cohen--Macaulay of dimension $d$ if and only if $X$ is Cohen--Macaulay of dimension $d$.
\end{lemma}

\begin{proof}
	Hatcher--Wahl show that $Y$ is Cohen--Macaulay of dimension $d$ if $X$ is, see \cite[Proposition 3.5]{hatcherwahl2010}. Assume that $Y$ is Cohen--Macaulay of dimension $d$ and consider the natural ``forgetting labels'' map $\pi\colon Y \to X$. Since $Y$ is a complete join complex over $X$, we must have $\dim(X) = d$. Picking a preimage $s(x) \in \pi^{-1}(x)$ for every vertex $x$ of $X$ yields a simplicial map $s\colon X \to Y$ such that $\pi \circ s = id_{X}$, i.e.\ it exhibits $X$ as a retract of $Y$. It follows that $X$ is $(d-1)$-connected. Let $\Delta$ be any simplex in $X$ and let $\Delta'$ be a simplex of $Y$ of dimension $\dim (\Delta)$ such that $\pi(\Delta') = \Delta$ (for example, $\Delta' = s(\Delta)$). Then $\Link_Y(\Delta')$ is a complete join complex over $\Link_{X}(\Delta)$, see e.g.\ the last paragraph of the proof of \cite[Proposition 3.5]{hatcherwahl2010}. In particular, $\Link_{X}(\Delta)$ is of dimension $(d - \dim(\Delta) - 1)$. By assumption $\Link_Y(\Delta')$ is $(d - \dim(\Delta') - 2)$-connected, and since $\Link_X(\Delta)$ is a retract of this space it has to be $(d - \dim(\Delta) - 2)$-connected.
\end{proof}

\begin{corollary}
	\label{corollary:partial-bases-is-a-complete-join-complex}
	If $R$ satisfies \cref{assumption-invariant-basis} and \cref{assumption-stably-free} of \cref{assumption:R}, then $\B_n(R)$ is a complete join complex over $\on{F}_n(R)$. In particular, $R$ satisfies \cref{assumption-partial-bases} of \cref{assumption:R} if and only if for all $n \geq 0$ the complex of frames $\on{F}_n(R)$ is Cohen--Macaulay.
\end{corollary}

\begin{proof}
	The map $\B_n(R) \to \on{F}_n(R)$ that sends a partial basis $\{\vec v_1, \dots, \vec v_k\}$ to the partial frame $\{\langle \vec v_1 \rangle, \dots, \langle \vec v_k \rangle\}$ exhibits $\B_n(R)$ as a complete join complex on $\on{F}_n(R)$. Indeed, in the language of \cite[Example 3.3]{hatcherwahl2010}, the set of labels of a vertex $L \in \on{F}_n(R)$ is exactly the set of unimodular vectors contained in $L$. The claim then follows from \cref{lemma:complete-join-complex}.
\end{proof}

We also record the following simple observation, which is frequently used in subsequent induction arguments.

\begin{corollary}
	\label{corollary:links-in-B-are-complete-join-complexes}
	Let $m \geq 0$ and $n \geq 1$ and view $R^n$ as a summand in $R^{m+n} = R^m \oplus R^n$. If $R$ satisfies \cref{assumption-invariant-basis} and \cref{assumption-stably-free} of \cref{assumption:R}, then $B^m_n(R)$ is a complete join complex over $B_n(R)$. In particular, $B^m_n(R)$ is Cohen--Macaulay of dimension $(n-1)$ for all $m \geq 0$ if and only if $\B_n(R)$ is Cohen--Macaulay of dimension $(n-1)$.
\end{corollary}

\begin{proof}
	If $\vec v \in R^{m+n}$, we denote by $\vec v' = \vec v - (\sum_{i = 1}^{m} a_i \vec e_i)$ the vector obtained by setting the first $m$ coordinates $(a_1, \dots, a_m)$ of $\vec v$ to zero. The map $\B_n^m(R) \to \B_n(R)$ that sends a partial basis $\{\vec v_1, \dots, \vec v_k\}$ to the partial basis $\{\vec v_1', \dots, \vec v_k'\}$ exhibits $\B_n^m(R)$ as a complete join complex on $\B_n(R)$. Indeed, in the language of \cite[Example 3.3]{hatcherwahl2010}, the set of labels of a vertex $\vec v' \in \B_n(R)$ is exactly the set of vectors $\{\vec v' + (\sum_{i = 1}^{m} a_i \vec e_i): a_i \in R\}$ whose last $n$ coordinates agree with those of $\vec v'$. The claim then follows from \cref{lemma:complete-join-complex}.
\end{proof}

In the next step, we relate the frame complexes of $R$ to the co-frame complexes of $R^{op}$.

\begin{lemma}
	\label{lemma:coframe-complex-is-cm}
	Let $R$ be a ring satisfying \cref{assumption:R} and let $M$ be a free $R$-module of rank $n$. Let $M^\vee = \Hom_R(M, R)$ be its dual $R^{op}$-module, which is free of rank $n$. The following map is a simplicial isomorphism:
	$$(-)^\circ : \on{F}(M) \to \on{coF}(M^\vee): \{L_1, \dots, L_k\} \mapsto \{L_1^\circ, \dots, L_k^\circ\},$$
	where $L_i^\circ$ is the $R^{op}$-module of $R$-linear functions $f\colon M \to R$ with $f|_{L_i} = 0$.
	In particular, $\on{coF}(M^\vee)$ is Cohen--Macaulay of dimension $(n-1)$.
\end{lemma}

\begin{proof}
	The map is well-defined: If $\{\vec v_1, \dots, \vec v_k\}$ is a partial basis such that $L_i = \langle \vec v_i \rangle$, we can extend it to a basis $\{\vec v_1, \dots, \vec v_n\}$ and can consider the dual basis $\{\vec v_1^*, \dots, \vec v_n^*\}$ of $M^\vee$. It then holds that $\vec v_j^* \in L_i^\circ$ if and only if $j \neq i$. Since $\rank L_i^\circ = n-1$ and $R^{op}$ satisfies \cref{ibn-III}, it follows that $L_i^\circ = \langle \vec v_j^* : j \neq i \rangle$ and hence that $\{L_1^\circ, \dots, L_k^\circ\}$ is indeed a simplex in $\on{coF}(M^\vee)$. The canonical isomorphism $M \cong (M^\vee)^\vee$ (see e.g.\ \cite[Theorem 9.2]{blyth1990}) identifies $L_i \cong (L_i^\circ)^\circ$ and induces a simplicial isomorphism $\on{F}(M) \cong \on{F}((M^\vee)^\vee)$. Since this isomorphism factors through $\on{coF}(M^\vee)$, it follows that the map is injective. Similarly, one can use the canonical isomorphism $M^\vee \cong ((M^\vee)^\vee)^\vee$ to check that the map is surjective. To conclude that the map is a simplicial isomorphism, we need to see that $\{L_1, \dots, L_k\}$ is a simplex in $\on{F}(M)$ if $\{L_1^\circ, \dots, L_k^\circ\}$ is a simplex in $\on{coF}(M^\vee)$. This holds because any basis witnessing that $\{L_1^\circ, \dots, L_k^\circ\}$ is a simplex in $\on{coF}(M^\vee)$ is the dual basis $\{\vec v_1^*, \dots, \vec v_n^*\}$ of some basis $\{\vec v_1, \dots, \vec v_n\}$ of $M$ (see e.g. \cref{lemma:dualizing-and-inverse-transpose}). If $L_i^\circ = \langle \vec v_j^* : j \neq i \rangle$, then it holds that $f(\vec v_i) = 0$ for every $f \in L_i^\circ$. Therefore $\vec v_i \in L_i$, hence $L_i = \langle \vec v_i \rangle$ and $\{L_1, \dots, L_k\}$ is a simplex in $\on{F}(M)$. This completes the proof.
\end{proof}

Let $R$ be a ring satisfying \cref{assumption:R} and $M$ be a free $R$-module of rank $n$. We note that there are two interesting poset maps:
The first one is a spanning map for frames,
$$\on{span}: \on{\mathbb{F}}(M)^{(n-2)} \to \bT(M): \{L_1, \dots, L_k\} \mapsto \oplus_1^k L_i,$$
which is $(n-2)$-connected by exactly the same argument as in \cref{lemma:span-map} using \cref{corollary:partial-bases-is-a-complete-join-complex}.
The second one is a co-spanning map for co-frames,
$$\on{cospan}: \on{\mathbb{coF}}(M^\vee)^{(n-2)} \to \bT(M^\vee)^{op}: \{C_1, \dots, C_k\} \mapsto \cap_1^k C_i.$$
Note that this map is well-defined: If $\{\vec v_1^*, \dots, \vec v_n^*\}$ is a basis of $M^\vee$ such that $C_i = \langle \vec v_j^* : j \neq i \rangle$, then it holds that $\cap_1^k C_i = \langle \vec v_j^* : j \notin \{1, \dots, k\} \rangle$ is a free summand of rank $n-k$ with complement $\langle \vec v_j^* : 1 \leq j \leq k \rangle$. The next lemma relates these two maps.

\begin{lemma}
	\label{lemma:cospanning-map-is-highly-connected}
	Let $R$ be a ring satisfying \cref{assumption:R} and let $M$ be a free $R$ module of rank $n$. Then the spanning and co-spanning map fit into a commutative diagram:
	\begin{center}
		\begin{tikzcd}
			\on{\mathbb{F}}(M)^{(n-2)} \arrow[r, "{(-)^{\circ}}", "\cong"'] \arrow[d, "\on{span}"] & \on{\mathbb{coF}}(M^\vee)^{(n-2)} \arrow[d, "\on{cospan}"]\\
			\bT(M) \arrow[r, "{(-)^{\circ}}", "\cong"'] & \bT(M^\vee)^{op}
		\end{tikzcd}
	\end{center}
	In particular, the map $\on{cospan}$ is $(n-2)$-connected.
\end{lemma}

\begin{proof}
	We only need to check that the diagram commutes. Consider a simplex $\{L_i\}_{i \in \{1, \dots, k\}}$ in $\on{\mathbb{F}}(M)^{(n-2)}$. There is a basis $\{\vec v_i\}_{i \in \{1, \dots, n\}}$ of $M$ such that $L_i = \langle \vec v_i \rangle$ for $1 \leq i \leq k$, and we write $\{\vec v_i^*\}_{i \in \{1, \dots, n\}}$ for its dual basis of $M^\vee$.	
	It follows that $\on{span}(\{L_i\}) = \langle \vec v_i : 1 \leq i \leq k \rangle$, and therefore that $\on{span}(\{L_i\})^\circ = \langle \vec v_i^* : i \notin \{1, \dots, k\} \rangle$.
	On the other hand, $(\{L_i\}_{i \in \{1, \dots, k\}})^\circ = \{L_i^\circ\}_{i \in \{1, \dots, k\}}$ with $L_i^\circ = \langle \vec v_j^* : j \neq i \rangle$,
	hence $\on{cospan}(\{L_i^\circ\}_{i \in \{1, \dots, k\}}) = \cap_{i \in \{1, \dots, k\}} L_i^\circ = \langle \vec v_i^* : i \notin \{1, \dots, k\} \rangle$ as well.	
	Since the diagram commutes, and claim about $\on{cospan}$ follows from \cref{lemma:comparing-the-tits-complexes-of-R-and-Rop} and \cref{lemma:coframe-complex-is-cm}.
\end{proof}

Following \cite{sadofschicosta2020}, we will now use the co-spanning map to study the connectivity of $\B(M^\vee)$ if $M^\vee$ has small rank. These base cases require extra care, and are used in a subsequent induction argument to establish $1$-connectedness. They are discussed in the next lemma, the proof of which is analogous to the argument for \cite[Proposition 5.9]{sadofschicosta2020}.

\begin{lemma}
	\label{lemma:relation-to-Rop-base-cases}
	Let $R$ be a ring satisfying \cref{assumption:R}, let $M$ be a free $R$-module of rank $n$ and consider its dual $R^{op}$-module $M^\vee$. Then $B(M^\vee)$ is Cohen--Macaulay of dimension $(n-1)$ if $n \in \{1, 2, 3\}$.
\end{lemma}

\begin{proof}
	For the first part: $B(M^\vee)$ is nonempty, i.e.\ $(-1)$-connected, if $n = 1$ since any basis $\{\vec v_1^*\} \subseteq M^\vee$ is a vertex. 
	
	For the second part: If $n = 2$, then \cref{lemma:coframe-complex-is-cm} implies that $\on{coF}(M^\vee)$ is $0$-connected. But $\on{coF}(M^\vee) = \on{F}(M^\vee)$ if $n = 2$, and therefore $\on{F}(M^\vee)$ is Cohen--Macaulay of dimension $1$ as well. By \cref{lemma:relation-to-Rop} and \cref{corollary:partial-bases-is-a-complete-join-complex}, it follows that $B(M^\vee)$ is a complete join complex over $\on{F}(M^\vee)$ and therefore \cref{lemma:complete-join-complex} implies that $\on{B}(M^\vee)$ is Cohen--Macaulay of dimension $1$ as well.

	For the third part, the argument is similar to \cite[Proposition 5.9]{sadofschicosta2020}: We start with several preliminary observations to establish the analogue of \cite[Proposition 5.8]{sadofschicosta2020} in our setting. If $n  = 3$, then \cref{lemma:coframe-complex-is-cm} implies that $\on{coF}(M^\vee)$ is Cohen--Macaulay of dimension $2$, i.e.\ this complex is, in particular, $1$-connected. Hence the set of boundaries of barycentric subdivisions of $2$-simplices in $\on{coF}(M^\vee)$ yield a set $\{\partial \Delta_i^2\}$ of $1$-loops (in the sense of \cite[Definition 5.1]{sadofschicosta2020}) with the $\pi_1$-spanning property (in the sense of \cite[Definition 5.2]{sadofschicosta2020}) for the barycentric subdivision of the $1$-skeleton $\on{\mathbb{coF}}(M^\vee)^{(1)}$. By \cref{lemma:cospanning-map-is-highly-connected}, the co-spanning map $\on{cospan}\colon \on{\mathbb{coF}}(M^\vee)^{(1)} \to \bT(M^\vee)^{op}$ is $1$-connected and therefore induces a surjection between the fundamental groups. It follows that the image of the set of $1$-loops $\{\partial \Delta_i^2\}$ in $\on{\mathbb{coF}}(M^\vee)^{(1)}$ gives a set of $1$-loops with the $\pi_1$-spanning property for $T(M^\vee)^{op}$. For each $2$-simplex $\Delta = \{C_1, C_2, C_3\}$ in $\on{coF}(M^\vee)$, there exists by definition a basis $\{\vec v_1^*, \vec v_2^*, \vec v_3^*\}$ of $M^\vee$ such that $C_i = \langle \vec v_j^* : j \neq i \rangle$. Notice that $\Delta' = \{\vec v_1^*, \vec v_2^*, \vec v_3^*\}$ is a $2$-simplex in $\B(M^\vee)$, and that the $1$-loop obtained by applying the spanning map $\on{span} \colon \bB(M^\vee)^{(1)} \to \bT(M^\vee)$ to the boundary of its barycentric subdivision is equal to that associated to the boundary of $\Delta = \{C_1, C_2, C_3\}$ using the co-spanning map, using that the geometric realizations satisfy $T(M^\vee) = T(M^\vee)^{op}$. This shows that the set $\{\gamma_i\}$ of $1$-loops in $\bB(M^\vee)^{(1)}$ obtained from the boundaries of barycentric subdivisions of $2$-simplices of $\B(M^\vee)$ has the property that applying the spanning map $\{\on{span}(\gamma_i)\}$ yields a set of $1$-loops that has the $\pi_1$-spanning property for $\bT(M^\vee)$.
	
	Now, we consider the spanning map $\on{span} \colon \bB(M^\vee)^{(1)} \to \bT(M^\vee)$. It follows from the previous two parts that the lower fibers $\on{span}_{\leq V^\circ} = \bB(V^\circ)$ of this map are $(\rank{V^\circ} - 2)$-connected and the codomain is Cohen--Macaulay of dimension $1$ by \cref{lemma:comparing-the-tits-complexes-of-R-and-Rop}. Therefore the map is $1$-spherical in the sense of \cite[Definition 2.4]{sadofschicosta2020}, and by the first part of \cite[Theorem 5.7]{sadofschicosta2020} it follows that its domain is $1$-spherical and, in particular, $0$-connected. But this means that $\bB(M^\vee)$ is also $0$-connected, and we are left with checking that it is simply connected. To see this, we apply the second part of \cite[Theorem 5.7]{sadofschicosta2020} to construct a set $\{\gamma_i\} \cup \{\alpha_i\} \cup \{\eta_{i,j}\}$ of $1$-loops with the $\pi_1$-spanning property for $\bB(M^\vee)^{(1)}$ such that each $1$-loop is null-homotopic in $\bB(M^\vee)$. For this, let $V^\circ \in \bT(M^\vee)$. 
	
	We already constructed the $1$-loops $\{\gamma_i\}$, and we note that these $1$-loops are null-homotopic in $\bB(M^\vee)$ as boundaries of $2$-simplices. 
	
	If $\rank V^\circ = 2$, we consider an arbitrary set $\{\alpha_i\}$ of $1$-loops in the graph $\bB(V^\circ)$ with the $\pi_1$-spanning property. Since $V^\circ$ is a summand of rank $2$ in $M^\vee$, there exists some unimodular vector $\vec w^*$ such that $V^\circ \oplus \langle \vec w^* \rangle = M^\vee$. It follows that each $1$-loop $\alpha_i$ is null-homotopic in $\B(M^\vee)$ because it is contained in the contractible subcomplex $\{\vec w^*\} \ast \B(V^\circ)$.
	
	If $\rank V^\circ = 1$, we consider some arbitrary set $\{\alpha_i\}$ of $0$-loops in the discrete set $\bB(V^\circ)$ with the $\pi_0$-spanning property (in the sense of \cite[Definition 5.1 and Definition 5.2]{sadofschicosta2020}), and an arbitrary set $\{\bar\beta_j\}$ of $0$-loops in the discrete set $\bT(M^\vee)_{> V^\circ}$ with the $\pi_0$-spanning property. We now describe how to choose the $1$-loop $\eta_{i,j}$ in $\bB_3(M^\vee)$ for each pair of indices $(i,j)$: If $\alpha_i = \{\vec x_1^*, \vec x_2^*\}$ and $\bar \beta_j = \{H_1, H_2\}$, there exist unimodular vectors $\vec z_1^*$ and $\vec z_2^*$ in $M^\vee$ such that $H_1 = V^\circ \oplus \langle \vec z_1^* \rangle$ and $H_2 = V^\circ \oplus \langle \vec z_2^* \rangle$. We then define $\eta_{i,j}$ to be the $1$-loop in $B(M^\vee)$ corresponding to the edge path $\vec x_1^* \rightsquigarrow \vec z_1^* \rightsquigarrow \vec x_2^* \rightsquigarrow \vec z_2^* \rightsquigarrow \vec x_1^*$. To see that $\eta_{i,j}$ is null-homotopic one observes that $\vec z_1^*$ and $\vec z_2^*$ are both vertices in $\Link_{B(M^\vee)}(\vec x_1^*) = \Link_{B(M^\vee)}(\vec x_2^*)$. The complex $\Link_{B(M^\vee)}(\vec x_1^*) = \Link_{B(M^\vee)}(\vec x_2^*)$ is isomorphic to the complex $\B_2^1(R^{op})$ and hence, by virtue of \cref{corollary:links-in-B-are-complete-join-complexes}, it is a complete join complex over a complex isomorphic to $B_2(R^{op})$. Since $B_2(R^{op})$ is Cohen--Macaulay of dimension $1$ by the second part of this lemma (proved above), \cref{corollary:links-in-B-are-complete-join-complexes} implies that $\Link_{B(M^\vee)}(\vec x_1^*) = \Link_{B(M^\vee)}(\vec x_2^*)$ is also Cohen--Macaulay of dimension $1$ and, in particular, $0$-connected. It follows that the $1$-loop $\eta_{i,j}$ is null-homotopic in $B(M^\vee)$.
	
	Now, by \cite[Theorem 5.7]{sadofschicosta2020} we have a set of $1$-loops in $\bB(M^\vee)^{(1)}$ with the $\pi_1$-spanning property such that each $1$-loop is null-homotopic in $\bB(M^\vee)$. Hence, $\bB(M^\vee)$ is $1$-connected. The connectivity property for links of simplices in $\B(M^\vee)$ follows from \cref{corollary:links-in-B-are-complete-join-complexes} and the first two parts of this lemma (proved above).
\end{proof}

We are now ready to prove the result in full generality. The proof of the next theorem is analogous to the proof of \cite[Theorem 6.2]{sadofschicosta2020}.

\begin{theorem}
	Let $R$ be a ring satisfying \cref{assumption:R}, let $M$ be free module of rank $n$ and consider its dual $R^{op}$-module $M^\vee$. Then $\bB(M^\vee)$ is Cohen--Macaulay of dimension $n-1$.
\end{theorem}

\begin{proof}
	We argue by induction on $n$. The base cases $n \in \{1, 2, 3\}$ have already been established in \cref{lemma:relation-to-Rop-base-cases}. For the induction step $n \geq 4$ we consider the spanning map $\on{span}\colon \bB(M^\vee)^{(n-2)} \to \bT(M^\vee)$. The codomain of this map is Cohen--Macaulay of dimension $(n-2)$ by \cref{lemma:comparing-the-tits-complexes-of-R-and-Rop}, and by the induction hypothesis $\on{span}_{\leq V^\circ} = \bB(V^\circ)$ is Cohen--Macaulay of dimension $(\rank(V^\circ) - 1)$ for every $V^\circ \in \bT(M^\vee)$. Therefore, it follows from the first part of \cite[Theorem 3.1]{sadofschicosta2020}, see also \cite[Theorem 9.1]{quillen1978}, that $\bB(M^\vee)^{(n-2)}$ is $(n-2)$-spherical. This means that $\bB(M^\vee)$ is $(n-3)$-connected, and in particular $1$-connected since $n \geq 4$. Hence, we only need to check that $\widetilde{H}_{n-2}(\bB(M^\vee)) = 0$. To see this, we want to apply the second part of \cite[Theorem 3.1]{sadofschicosta2020} to construct a basis $\{\gamma_i\} \cup \{\alpha_i \ast \beta_j\}$ of $\widetilde{H}_{n-2}(\bB(M^\vee)^{(n-2)}) = \ker(\partial_{n-2})$ such that each basis element is a boundary in $\bB(M^\vee)$ and therefore $\widetilde{H}_{n-2}(\bB(M^\vee)) = \ker(\partial_{n-2})/\im(\partial_{n-1}) \cong 0$.
	
	We first check that the three conditions stated in \cite[Theorem 3.1]{sadofschicosta2020} are satisfied: The first condition states that $\langle \sigma_1 \rangle \subseteq \langle \sigma_2 \rangle$ implies that $\Link_{\bB(M^\vee)^{(n-2)}}(\sigma_2) \subseteq \Link_{\bB(M^\vee)^{(n-2)}}(\sigma_1)$. This holds because \cref{lemma:relation-to-Rop}, \cref{lemma:summand-property} and the assumption imply that $\langle \sigma_2 \rangle = \langle \sigma_1 \rangle \oplus C$ for some free $R^{op}$-module $C$ and if $\tau \in \Link_{\bB(M^\vee)^{(n-2)}}(\sigma_2)$, then $\langle \tau \rangle \oplus \langle \sigma_2 \rangle = \langle \tau \rangle \oplus \langle \sigma_1 \rangle \oplus C$  is a free summand of $M^\vee$, so $\langle \tau \rangle \oplus \langle \sigma_1 \rangle$ is a free summand of $M^\vee$ (with free complement) and therefore $\tau \in \Link_{\bB(M^\vee)^{(n-2)}}(\sigma_1)$. The second condition states that if $\langle \sigma_1 \rangle \subseteq \langle \sigma_2 \rangle$ and $\langle \tau_1 \rangle \subseteq \langle \tau_2 \rangle$, then $\langle \sigma_1 \cup \tau_1 \rangle \subseteq \langle \sigma_2 \cup \tau_2 \rangle$ assuming $\sigma_1 \cup \tau_1$ and $\sigma_2 \cup \tau_2$ are simplices in $\bB(M^\vee)^{(n-2)}$. This is obviously true. The third condition states that for every $V^\circ \in \bT(M^\vee)$ and $\sigma \in \bB(M^\vee)^{(n-2)}$ with $\langle \sigma \rangle = V^\circ$, it holds that the map obtained by restriction $\bB(M^\vee)^{(n-2)}_{> \sigma} \to \bT(M^\vee)_{> V^\circ}$ is a surjection on the top homology. It follows from \cref{lemma:comparing-the-tits-complexes-of-R-and-Rop} and the fact that $\bT(M^\vee)_{> V^\circ} \cong \bT(M^\vee/V^\circ)$, see \cref{item-upper-fiber-tits-complex} of \cref{upperlowerTits}, that the codomain is Cohen--Macaulay. The induction hypothesis and \cref{corollary:links-in-B-are-complete-join-complexes} imply that $\on{span}_{\leq W^\circ} \cong \bB_{\rank(W^\circ) - \rank(V^\circ)}^{\rank(V^\circ)}(R^{op})$ is also Cohen--Macaulay. Therefore, the first part of \cite[Theorem 3.1]{sadofschicosta2020} shows that the third condition does indeed hold. We may hence apply the second part of \cite[Theorem 3.1]{sadofschicosta2020} to construct a basis for $\widetilde{H}_{n-2}(\bB(M^\vee)^{(n-2)})$.
	
	We now apply the second part of \cite[Theorem 3.1]{sadofschicosta2020} to construct a basis for $\widetilde{H}_{n-2}(\bB(M^\vee)^{(n-2)})$, and check that each of the elements in this basis is trivial in $\widetilde{H}_{n-2}(\bB(M^\vee))$: By \cref{lemma:comparing-the-tits-complexes-of-R-and-Rop}, we can choose the classes $\{\gamma_i\}$ in $\widetilde{H}_{n-2}(\bB(M^\vee)^{(n-2)})$ to be a collection of boundaries of barycentric subdivisions of top-dimensional simplices in $\B(M^\vee)$ with the property that $\{\on{span}(\gamma_i)\}$ is a basis of $\widetilde{H}_{n-2}(\bT(M^\vee)) = \St(M^\vee)$ consisting of apartment classes. In particular, $\gamma_i$ is a boundary in $\bB(M^\vee)$ and hence it holds that $\gamma_i = 0 \in \widetilde{H}_{n-2}(\bB(M^\vee))$. Now let $V^\circ \in \bT(M^\vee)$ and $\sigma \in \bB(M^\vee)^{(n-2)}$ with $V^\circ = \langle \sigma \rangle$ be arbitrary. Let $\{\alpha_i\}$ be a basis of $\widetilde{H}_{\rank(V^\circ) - 1}(\on{span}_{\leq V^\circ})$ and let $\{\beta_j\}$ be a collection of classes in $\widetilde{H}_{(n-2) - (\rank(V^\circ) - 1) - 1}(\Link_{\bB(M^\vee)^{(n-2)}}(\sigma))$ such that $\{\on{span}(\beta_j) \oplus V^\circ\}$ is a basis of $\widetilde{H}_{(n-2) - (\rank(V^\circ) - 1) - 1}(\bT(M^\vee)_{> V^\circ})$. We need to see that the classes $\{\alpha_i \ast \beta_j\}$ in $\widetilde{H}_{n-2}(\on{span}_{\leq V^\circ} \ast \Link_{\bB(M^\vee)^{(n-2)}}(\sigma))$ are trivial in $\widetilde{H}_{n-2}(\bB(M^\vee))$. By the induction hypothesis and \cref{corollary:links-in-B-are-complete-join-complexes}, it holds that $\Link_{\B(M^\vee)}(\sigma) \cong \B^{\rank(V^\circ)}_{n - \rank(V^\circ)}(R^{op})$ is Cohen--Macaulay of dimension $(n - \rank(V^\circ) - 1)$. Therefore, the classes $\beta_j$ are trivial in $\widetilde{H}_{(n-2) - (\rank(V^\circ) - 1) - 1}(\Link_{\bB(M^\vee)}(\sigma))$. Hence, there exists a chain $\omega_j$ in $C_{n-\rank(V^\circ) - 1}(\Link_{\bB(M^\vee)}(\sigma))$ such that $\partial_{n - \rank(V^\circ) - 1}(\omega_j) = (-1)^{|\alpha_i|} \beta_j$ and it follows that $\alpha_i \ast \beta_j$ is a boundary: $\partial_{n-1} (\alpha_i \ast \omega_j) = \alpha_i \ast \beta_j$. This completes the proof.	
\end{proof}

\subsection{Relating Charney's splitting complex to the $E_1$-splitting complex}
In this part of the appendix, we spell out the relation between the splitting complex $\bS_n^{E_1}(R)$ introduced in \cref{definition:e1-splitting-complex} and the $E_1$-splitting complex introduced in \cite[Definition 17.9]{galatiuskupersrandalwilliams2021cells} to make this work self-contained. Similar complexes and this relation have been studied by Hepworth in \cite{hepworth2020} for principle ideal domains. All of the arguments presented here are easy adaptations of results proved in \cite{hepworth2020}, and we do not claim any novelty.

\begin{lemma}
	\label{lemma:splitting-complex-vs-splitting-poset}
	Let $R$ be a ring satisfying \cref{assumption:R}. For every $n \in \N$ Charney's splitting complex $S_n^{E_1}(R)$ is $\GL_n(R)$-equivariantly isomorphic to the $n$-th $E_1$-splitting complex associated to the $E_\infty$-algebra $\on{BGL}(R)$ by \cite[Definition 17.9]{galatiuskupersrandalwilliams2021cells}.
\end{lemma}

\begin{proof}
	As explained in \cite[Example 2.4]{hepworth2020}, the family $\{\GL_n(R)\}_{n \in \N}$ is a family of groups with multiplication in the sense of \cite[Definition 2.1]{hepworth2020} by virtue of the symmetric monoidal structure $(\GL(R), \oplus, 0)$ introduced in \cref{subsec:e1-homology-and-charneys-splitting-complex}. Replacing the PID assumption in \cite[Proposition 3.4]{hepworth2020} by \cref{assumption-invariant-basis} and \cref{assumption-stably-free} of \cref{assumption:R} and the poset $S_R(R^n)$ introduced in the paragraph before \cite[Proposition 3.4]{hepworth2020} by $\bS_n^{E_1}(R)$ as in \cref{definition:e1-splitting-complex}, the proof of \cite[Proposition 3.4]{hepworth2020} applies to show that the splitting poset $\on{SP}_n$ associated to the family $\{\GL_n(R)\}_{n \in \N}$, see \cite[Definition 3.1]{hepworth2020}, is isomorphic to $\bS_n^{E_1}(R)$: the first time \cite{hepworth2020} uses that $R$ is a PID the claim holds in our setting by \cref{definition:e1-splitting-complex}, and the second time \cite{hepworth2020} uses that $R$ is a PID the claim follows from \cref{lemma:intersection-yields-complement} in our setting. By \cite[Proposition 6.4]{hepworth2020}, the semi-simplicial nerve of splitting poset $\on{SP}_n$ agrees with the $E_1$-splitting complex of \cite[Definition 17.9]{galatiuskupersrandalwilliams2021cells} using \cite[Remark 17.11]{galatiuskupersrandalwilliams2021cells}, compare with \cite[Section 1.2]{hepworth2020}.
\end{proof}

\subsection{Standard connectivity estimate and coinvariants of the Charney module}
\label{appendix:standard-connectivity-estimate-and-coinvariants-of-the-charney-module}
In this part of the appendix, we spell out the details of the proof of \cref{proposition:computational-input-split-steinberg} to make this work self-contained. The argument we present and the notions defined in this context are adaptations of the proof of \cite[Theorem 4.8]{kupersmillerpatzt2022} and the definitions in \cite[Section 4.3]{kupersmillerpatzt2022} to our setting.

\begin{definition}
	Let $V$ and $W$ be nonzero free summands of a free $R$-module $M$. Then $\bS^{E_1}(\cdot \subseteq V,W \subseteq \cdot|M)$ is the subposet of
	$\bS^{E_1}(M)$ of splittings $(U,T)$ such that $U\subseteq V$ and $W\subseteq T$. We
	write $\bS^{E_1}(\cdot, W \subseteq \cdot|M)$ when $V=M$ and
	$\bS^{E_1}(\cdot \subseteq V,\cdot|M)$ when $W=\{0\}$.
\end{definition}

We show that these $E_1$-splitting complexes are spherical using a technique due to Charney \cite{charney1980} called \emph{cutting down}. This is encapsulated in the following analogue of \cite[Lemma 4.10]{kupersmillerpatzt2022}.

\begin{lemma}\label{cuttingDown}
	Let $R$ be a ring satisfying \cref{assumption:R}, and let $M$ be a free	$R$-module. If $V,W$ are free summands of $M$ such that $V\cap W=\{0\}$ and $V\oplus W\subsetneq M$ is a proper summand, then for any complement $C$ of $W$ in $M$ containing $V$, we have an isomorphism
	$$\bS^{E_1}(\cdot\subseteq V,W\subseteq \cdot|M)\cong \bS^{E_1}(\cdot\subseteq V,\cdot|C).$$
	Moreover, this isomorphism is equivariant with respect to the subgroup of $\GL(M)$ that preserves $C$ and $V$ and fixes $W$ pointwise.
\end{lemma}

\begin{proof}
	We start by noting that $C$ is necessarily free by \cref{lemma:complements-are-free}, and that the containment $V \subsetneq C$ is necessarily proper (otherwise it would hold that $V\oplus W = M$, a contradiction). We claim that
	\begin{align*}
		\bS^{E_1}(\cdot\subseteq V,W\subseteq\cdot|M) &\leftrightarrow \bS^{E_1}(\cdot\subseteq V,\cdot|C)\\
		(U,T)& \mapsto (U,T\cap C)\\
		(U', T' \oplus W) &\mapsfrom (U',T')
	\end{align*}
	still defines a pair of \emph{well-defined} mutually inverse functors in our setting (exactly as in \cite[Lemma 4.10]{kupersmillerpatzt2022}):
	Firstly, observe that if $(U,T)\in \bS^{E_1}(\cdot \subseteq V,W\subseteq \cdot|M)$, then $U \subseteq V \subsetneq C$ and $W \subseteq T$. Hence it follows from \cref{lemma:intersection-yields-complement} that $U \oplus (T \cap C) = C$ and that $T \cap C$ is free and nonzero of rank $\rank(C) - \rank(U) > 0$, because $V$ is a proper summand of $M$.
	Secondly, the inverse to this functor sends a splitting $(U',T')$ of $C$ to $(U',T'\oplus W)$, which is a free splitting of $M$ since $C\oplus W\cong M$ and $U',T',W$ are all free.
	
	If $\phi\in\GL(M)$ preserves $C$, we must have that $C \cdot \phi = C$ (since $M = C \oplus W$, $M = (C \cdot \phi) \oplus (W \cdot \phi)$ and $C \cdot \phi \subseteq C$, \cref{lemma:summand-property} implies that $C = (C \cdot \phi) \oplus K$, and $K = 0$ because \cref{ibn-III} holds by \cref{assumption:R} and $\rank(C) = \rank(C \cdot \phi)$. Thus, $\phi$ may be restricted to an element of $\GL(C)$. Therefore, we have an action of the subgroup $\GL(M,\pres{C},\pres{V},\fix{W})$ of $\GL(M)$, consisting of maps that preserve $C$ and $V$ and fix $W$ pointwise, on both complexes. The fact that the isomorphism defined above is equivariant to this action follows from the fact that if $\phi\in \GL(M,\pres{C},\pres{V},\fix{W})$, then $(C \cdot \phi) = C$ and $(W \cdot \phi)=W$, so $(T \cap C) \cdot \phi =(T \cdot \phi)\cap C$ for any summand $T$ of $M$, and $(T' \oplus W) \cdot \phi =(T' \cdot \phi) \oplus W$ for any summand $T'$ of $C$.
\end{proof}

Since we did not assume that $R$ is commutative, we additionally need the following technical lemma relating certain relative splitting complexes associated to $R$ and $R^{op}$.

\begin{lemma}
	\label{lemma:dualizing-argument-R-vs-Rop}
	Let $R$ be a ring satisfying \cref{assumption:R}, let $C$ be a free $R$-module of finite rank and let $V$ be a free summand of $C$. Consider the associated free $R^{op}$-module $C^\vee = \Hom_R(C, R)$, which has the same rank as $C$, and its submodule $V^\circ$ of $R$-linear functions vanishing on $V$. Mapping a splitting $(U,T)$ to $(T^\circ, U^\circ)$ yields an isomorphism
	$$\bS^{E_1}(\cdot\subseteq V,\cdot|C) \cong \bS^{E_1}(\cdot,V^{\circ}\subseteq \cdot |C^{\vee}),$$
	which is compatible with the action of the subgroup $\GL(C, \pres V, \fix C/V) \leq \GL(C)$ on the left side and the action of the subgroup $\GL(C^\vee, \fix V^\circ) \leq \GL(C^\vee)$ on the right side under their identification obtained by restricting the inverse-transpose group isomorphism $\GL(C) \to \GL(V^\vee): \phi \mapsto (\phi^{-1})^*$ (compare \cref{lemma:dualizing-and-inverse-transpose}).
\end{lemma}

\begin{proof}
	We first show that the mapping $(U,T)\mapsto (T^{\circ}, U^{\circ})$ is	well-defined: That $T^{\circ} \cap U^{\circ} = \{0\}$ follows immediately from	the fact that $C = U\oplus T$, since any map vanishing on both $U$ and $T$
	must then vanish on all of $C$. Furthermore, restricting $f\colon C = U \oplus T \to R$ to $f|_T\colon T \to R$ leads to an exact sequence of $R^{op}$-modules $0 \to T^\circ \to C^\vee \to T^\vee \to 0$. This sequence is split, since $T^\vee$ is a free $R^{op}$-module and of the same rank as $T$ (see \cref{lemma:dualizing-and-inverse-transpose}). Because $C^\vee$ is also free of the same rank as $C$, it follows that $T^{\circ}$ is a free $R^{op}$-module of the same rank as $U$. Similarly, $U^{\circ}$ is a free $R^{op}$-module of the same rank as $T$. Hence, $(T^{\circ}, U^{\circ})$ is a splitting of 	$C^{\vee}$ into nonzero free summands. Furthermore, $U \subseteq V$ implies $V^{\circ} \subseteq U^{\circ}$ and, similarly, $(U_1, T_1)\leq (U_2,T_2)$ implies $(T_1^{\circ}, U_1^{\circ})\leq (T_2^{\circ}, U_2^{\circ})$. Thus, 
	$$\on{\mathcal{D}}\colon \bS^{E_1}(\cdot \subseteq V, \cdot | C) \to \bS^{E_1}(\cdot, V^{\circ}\subseteq \cdot |C^{\vee})$$
	is indeed a well-defined map of posets.
	
	We check that $\on{\mathcal{D}}$ is a poset isomorphism: Firstly, we show that $\on{\mathcal{D}}$ is bijective. We note that the canonical isomorphism $C \cong (C^\vee)^\vee$ (see e.g.\ \cite[Theorem 9.2]{blyth1990}) identifies $U$ and $(U^\circ)^\circ$ and induces an isomorphism of posets $\bS^{E_1}(V \subseteq \cdot,\cdot | C) \cong \bS^{E_1}((V^\circ)^\circ \subseteq \cdot,\cdot |(C^{\vee})^{\vee})$, which factors through $\on{\mathcal{D}}$. This implies that $\on{\mathcal{D}}$ is injective. Similarly, we see that $\on{\mathcal{D}}$ is surjective using $C^\vee \cong ((C^\vee)^\vee)^\vee$. Secondly, we show that $\on{\mathcal{D}}$ is an order embedding, i.e.\ that $(U_1, T_1)\leq (U_2,T_2)$ if and only if $(T_1^{\circ}, U_1^{\circ})\leq (T_2^{\circ}, U_2^{\circ})$. It remains to see that $(T_1^{\circ}, U_1^{\circ})\leq (T_2^{\circ}, U_2^{\circ})$ implies $(U_1, T_1)\leq (U_2,T_2)$. This follows from the argument in \cref{lemma:comparing-the-tits-complexes-of-R-and-Rop}, which shows that for any free summands $V_1,V_2$, if $V_2^{\circ}\subseteq V_1^{\circ}$, then $V_1\subseteq V_2$. Thus, $\on{\mathcal{D}}$ is a poset isomorphism.
	
	Finally, we check that the isomorphism $\on{\mathcal{D}}$ is compatible with suitable group actions: First note that if $\phi \in \GL(C,\pres{V},\fix{C/V})$, then $f \cdot (\phi^{-1})^* = f \circ \phi^{-1} = f$ for all $f\in V^{\circ}$. We can see this, for example, by choosing a complement $W$ of $V$. Then we can write any $c\in C$ as $c=v+w$ for $v\in V$ and $w\in W$, and $f(c)=f(w)$. Since $\phi$ fixes $C/V$ and preserves $V$, we have that $\phi^{-1}(v+w) =v' + w$, for some $v'\in V$. Thus, $f(\phi^{-1}(v+w))= f(v'+w)=f(w)=f(c)$. Therefore, the isomorphism $\GL(C)\to \GL(C^{\vee})$ restricts to a map from the subgroup $\GL(C,\pres{V},\fix{C/V})$ to the subgroup $\GL(C^{\vee}, \fix{V^{\circ}})$. We claim that this is an isomorphism, i.e.\ that this restriction is surjective: Every $(\phi^{-1})^* \in \GL(C^{\vee}, \fix{V^{\circ}})$ satisfies $f = f \circ \phi$ for all $f \in V^\circ$. For $\vec v \in V$, it therefore holds that $\on{ev}_{\phi(\vec v)} \in (V^\circ)^\circ$ because $\on{ev}_{\phi(\vec v)}(f) = f(\phi(\vec v)) = 0$ for all $f \in V^\circ$. The canonical isomorphism $C \cong (C^\vee)^\vee$ identifies $V$ with $(V^\circ)^\circ$, and hence implies that $\phi(v) \in V$, i.e.\ that $\phi$ preserves $V$. We next observe that for every $\vec c \in C$ it holds that $\on{ev}_{\vec c - \phi(\vec c)}(f) = 0$ for all $f \in V^\circ$ and therefore $\on{ev}_{\vec c - \phi(\vec c)} \in (V^\circ)^\circ$. The canonical isomorphism $V \cong (V^\vee)^\vee$ hence implies that $\vec c - \phi(\vec c) \in V$, i.e.\ that $\phi$ fixes $C/V$.
	
	Suppose $\phi \in \GL(C)$ takes $(U_1,T_1)$ to $(U_2, T_2)$. Then $\phi^{-1}$ maps	$(U_2, T_2)$ to $(U_1, T_1)$. In particular, the preimage of $U_1$ under	$\phi^{-1}$ is $U_2$, and $\phi^{-1}(U_2)=U_1$, so $f \in U_1^{\circ}$ if and only if $f \cdot (\phi^{-1})^* = f\circ \phi^{-1} \in U_2^{\circ}$. Similarly, $f\in T_1^{\circ}$ if and only if $f \cdot (\phi^{-1})^* = f\circ \phi^{-1}\in	T_2^{\circ}$. Thus, $(\phi^{-1})^*$ takes $(T_1^{\circ},U_1^{\circ})$ to $(T_2^{\circ}, U_2^{\circ})$. This establishes the compatibility of our isomorphism with the group action.
\end{proof}

The next proposition is the analogue of the first part of \cite[Theorem 4.8]{kupersmillerpatzt2022} and shows that Charney's connectivity result \cite[Theorem 1.1]{charney1980} extends to our setting. It relies on a map-of-posets argument which is analogous to that used to prove \cite[Theorem 4.3]{galatiuskupersrandalwilliams2018cellsandfinite} and \cite[Theorem 4.8]{kupersmillerpatzt2022}.

\begin{proposition}\label{freeCharney}
	Assume that $R$ satisfies \cref{assumption:R}. Then $\bS_n^{E_1}(R)$ is $(n-2)$-spherical for all $n$, and if $W \in \bT_n(R)$ is a proper nonzero free summand of $R^n$, then $\bS^{E_1}(\cdot,W\subseteq \cdot|R^n)$ is $(n-\rank{W}-1)$-spherical.
\end{proposition}
	
\begin{proof}	
	This argument exploits that $R^{op}$ also satisfies \cref{assumption:R} by the virtue of \cref{theorem:relation-to-Rop}. All steps in this proof are carried out for $R$ and $R^{op}$ simultaneously, even though we only write out the details for $R$.
	
	We first establish the second part. If $\rank{W}=n-1$, then \cref{ibn-III} (established in \cref{corollary:assumption-consequence-1}) implies that there is no summand $W'$ of $R^n$ satisfying $W \subsetneq W' \subsetneq R^n$, and hence $\bS^{E_1}(\cdot,W\subseteq \cdot) \cong \bT(R^n,W) \cong \bT_{n-\rank{W}}^{\rank{W}}(R)$, which is $(n-\rank{W}-1)$-spherical by \cref{lemma:span-map}. Thus, it suffices to prove the result for $1\leq \rank{W}\leq n-2$. We do this by induction on $n$ with base case $n=2$.
	
	If $n=2$, then $\rank{W}=n-1$ is the only possibility, so there is nothing left to prove. For the induction step, we consider the map of posets
	$$f\colon\bS^{E_1}(\cdot, W \subseteq \cdot|R^n)\to \bT(R^n,W)$$
	sending a splitting $(V,U)$ to $V$, exactly as in the proof of \cite[Theorem 4.8]{galatiuskupersrandalwilliams2018cellsandfinite} and \cite[Theorem 4.8]{kupersmillerpatzt2022}. Note that this map is well-defined in our setting: If $V \oplus U = R^n$, $L \oplus W = R^n$ and $W \subseteq U$, then \cref{lemma:summand-property} ensures that $U = L' \oplus W$ for some free $L'$. It follows that $W$ admits a complement $C$ in $R^n$ that contains $V$ as a summand, namely $C = L' \oplus V$. Therefore it does indeed hold that $V \in \bT(R^n, W)$, and that the map is well-defined. The following claims about this poset map hold:
	\begin{enumerate}
		\item \label[claim]{claim-1} The target is $(n-\rank{W}-1)$-spherical, since $\bT(R^n,W) \cong	\bT_{n-\rank{W}}^{\rank{W}}(R)$ and invoking \cref{lemma:span-map};
		\item \label[claim]{claim-2} $\bT(R^n,W)_{>V}$ is $(n-\rank{W}-\rank{V}-1)$-spherical if $V\in\bT(R^n,W)$, since $\bT(R^n,W)_{>V} \cong \bT_{n-\rank{W}-\rank{V}}^{\rank{W}}(R)$ by \cref{upperlowerRelTits} and invoking \cref{lemma:span-map};
		\item \label[claim]{claim-3} $\bS^{E_1}(\cdot, W \subseteq \cdot|R^n)_{f\leq V}$ is $(\rank{V}-1)$-spherical for every $V \in \bT(R^n, W)$, as we shall see now.
	\end{enumerate}
	To show \cref{claim-3}, we let $V \in \bT(R^n,W)$ and consider two cases. If $V\oplus W=R^n$, then $(V,W)$ is a terminal object in $\bS^{E_1}(\cdot, W \subseteq \cdot|R^n)_{f\leq V}$ and the complex is contractible. If $V\oplus W \neq R^n$, we apply cutting down: By the definition of $\bT(R^n, W)$, it holds that $V$ is a summand of a complement $C$ of $W$. In particular, $V \oplus W \subsetneq R^n$ is a proper summand of $R^n$ in this second case. Using that $V \subseteq C$, \cref{cuttingDown} therefore yields an isomorphism
	$$\bS^{E_1}(\cdot\subseteq V,W\subseteq |R^n) \cong \bS^{E_1}(\cdot\subseteq V,\cdot|C).$$
	Consider the free $R^{op}$-module $C^{\vee}=\Hom_R(C,R)$, which is of the same rank as $C$ (see \cref{lemma:dualizing-and-inverse-transpose}) and its submodule $V^\circ$ of $R$-linear functions that vanish on $V$. \cref{lemma:dualizing-argument-R-vs-Rop} yields a poset isomorphism $\bS^{E_1}(\cdot\subseteq V,\cdot|C) \cong \bS^{E_1}(\cdot,V^{\circ}\subseteq \cdot |C^{\vee})$.
	Since $C^{\vee}$ is a free $R^{op}$-module of rank $n-\rank{W}$ (which is strictly less than $n$), and $\rank{V^{\circ}}=\rank{C}-\rank{V}$, it follows from the induction hypothesis for $R^{op}$ that $\bS^{E_1}(\cdot,V^{\circ}\subseteq \cdot |C^{\vee})$ is
	$(\rank{C} -(\rank{C} -\rank{V})-1)=(\rank{V}-1)$-spherical. This finishes the proof of \cref{claim-3}.
	
	\cref{claim-1}, \cref{claim-2} and \cref{claim-3} show that we can apply \cite[Theorem 4.1]{galatiuskupersrandalwilliams2018cellsandfinite} with $t(V) = n - \rank{W} - \rank{V}$ to conclude that the source $\bS^{E_1}(\cdot,W\subseteq \cdot|R^n)$ is indeed $(n-\rank{W}-1)$-spherical.
	
	Now we prove the statement about connectivity of $\bS^{E_1}_n(R)$ and consider the map of posets
	$$f\colon \bS^{E_1}_n(R)\to \bT_n(R)^{op}$$
	sending a splitting $(V,W)$ to $W$,  exactly as in the proof of \cite[Theorem 4.8]{galatiuskupersrandalwilliams2018cellsandfinite} and \cite[Theorem 4.8]{kupersmillerpatzt2022}. Note that this map is also well-defined in our setting, since $W$ is indeed a nontrivial proper summand of $R^n$. The following claims hold:
	\begin{enumerate}
		\item \label[claim]{claim-1-part-2} The target $\bT_n(R)^{op}$ is $(n-2)$-spherical by \cref{lemma:span-map};
		\item \label[claim]{claim-2-part-2} $\bT_n(R)^{op}_{>W}$ is $(\rank{W}-2)$-spherical if $W\in \bT_n(R)$, since $\bT_n(R)^{op}_{>W}=\bT_n(R)_{<W}\cong \bT_{\rank{W}}(R)$ by \cref{upperlowerTits} and invoking \cref{lemma:span-map};
		\item \label[claim]{claim-3-part-2} $\bS^{E_1}_n(R)_{f\leq W}\cong \bS^{E_1}(\cdot,W\subseteq\cdot|R^n)$ is $(n-\rank{W}-1)$-spherical as we have argued above.
	\end{enumerate}
	\cref{claim-1-part-2}, \cref{claim-2-part-2} and \cref{claim-3-part-2} show that we can apply \cite[Theorem 4.1]{galatiuskupersrandalwilliams2018cellsandfinite} with $t(W) = \rank{W} - 1$ to conclude that the source $\bS^{E_1}_n(R)$ is indeed $(n-2)$-spherical.
\end{proof}

In light of \cref{freeCharney}, it makes sense to define relative versions of the Charney module $\on{Ch}_n(R)$, see \cref{def:charney-module}, if $R$ satisfies \cref{assumption:R}.

\begin{definition}
	Let $R$ be a ring such that $S^{E_1}(\cdot,W \subseteq \cdot|R^n)$ is $(n-\rank{W}-1)$-spherical for $W \in \bT_n(R)$. We define the \emph{relative Charney module} to be the right $\GL(R^n,\fix{W})$-module
	$$\on{Ch}(R^n,W) \coloneqq \widetilde H_{n-\rank{W}-1}(S^{E_1}(\cdot,W\subseteq\cdot|R^n); \Z).$$
\end{definition}

The last proposition in this appendix is the analogue of the second part of \cite[Theorem 4.8]{kupersmillerpatzt2022} and shows that their vanishing theorem for the coinvariants of Charney modules also holds to our setting. This builds on the map-of-posets arguments carried out in the proof of the previous proposition.

\begin{proposition}
	\label{SSTvanish}
	Assume that $R$ satisfies \cref{assumption:R}, and let $\K$ be a commutative ring such that $2 \in \K^\times$ is a unit and $W \in \bT_n(R)$. If $n \geq 2$, then
	$$(\on{Ch}_n(R)\otimes \K)_{\GL_n(R)} = 0, \text{ and }	(\on{Ch}(R^n,W)\otimes \K)_{\GL(R^n, \fix{W})}=0.$$
\end{proposition}

\begin{proof}
	This argument runs parallel to the proof of \cref{freeCharney} and uses facts established there.
	Again, we use that $R^{op}$ also satisfies \cref{assumption:R} by virtue of \cref{theorem:relation-to-Rop}. And again, all steps in this proof are carried out for $R$ and $R^{op}$ simultaneously, even though we only write out the details for $R$. 
	
	We start by establishing the second part of the claim. If the rank of $W$ is $n-1$, then $\on{Ch}(R^n,W)$ is isomorphic to $\on{St}(R^n,W)$ by the proof of \cref{freeCharney} (the isomorphism $\bS^{E_1}(\cdot,W\subseteq \cdot) \cong \bT(R^n,W)$ is $\GL(R^n,\fix{W})$-equivariant). In this case, the result has therefore already been shown
	in \autoref{corollary:coinvariants-of-relative-steinberg}. Thus, it suffices to prove the result for $1\leq \rank{W}\leq n-2$. We do this by induction on $n$ with base case $n=2$.
	
	If $n=2$, then $\rank{W}=n-1$ is the only possibility, so there is nothing left to prove. For the induction step, the first map-of-posets argument in the proof of \cref{freeCharney} yields, by the second part of \cite[Theorem 4.1]{galatiuskupersrandalwilliams2018cellsandfinite}, a filtration of $\K[\GL(R^n,W)]$-modules $0=F_{n- \rank{W}}\subseteq \dots\subseteq F_{-1}=\on{Ch}(R^n,W)\otimes \K$ such that
	$$F_{-1}/F_0\cong \St(R^n,W)\otimes \K,$$
	and if $q\geq 0$,
	\begin{align*}
		F_q/F_{q+1}&\cong\bigoplus_{\substack{V\in \bT(R^n,W)\\ \rank{V}=q+1}} \St(R^n/V,W)\otimes \widetilde{H}_{q}(\bS^{E_1}(\cdot\subseteq V,W\subseteq \cdot|R^n))\otimes \K\\
		&\cong \Ind_{\GL(R^n,\fix{W},\pres{V})}^{\GL(R^n,\fix{W})}\St(R^n/V,W)\otimes \widetilde{H}_{q}(\bS^{E_1}(\cdot\subseteq V,W\subseteq \cdot|R^n))\otimes \K,
	\end{align*}
	where $V$ is a choice of rank $q+1$ summand in $\bT(R^n,W)$, and $\GL(R^n,\fix{W},\pres{V})$ denotes the subgroup of $\GL_n(R)$ that fixes $W$ pointwise and maps $V$ to $V$. It then suffices to show that the $\GL(R^n,\fix{W})$-coinvariants of $F_q/F_{q+1}$ vanish for all $q$.
	
	When $q=-1$, we have that
	$$(\St(R^n,W)\otimes \K)_{\GL(R^n,\fix{W})}=0$$
	by \cref{corollary:coinvariants-of-relative-steinberg}.	If $q\geq 0$, then Shapiro's Lemma says that the $\GL(R^n,\fix{W})$-coinvariants	of $F_q/F_{q+1}$ are isomorphic to the $\GL(R^n,\fix{W},\pres{V})$-coinvariants	of
	$$\St(R^n/V,W)\otimes\widetilde{H}_{q}(\bS^{E_1}(\cdot\subseteq V,W\subseteq\cdot|R^n))\otimes \K$$
	for some choice of $V$. To show that these coinvariants vanish, we restrict to the subgroup $K$ defined by the short exact sequence
	$$0\to K\to \GL(R^n,\fix{W},\pres{V})\to \GL(R^n/V,\fix{W})\to 0.$$
	In particular, every element of $K$ acts trivially on $R^n/V$, so $K$ acts trivially on $\St(R^n/V,W)$. Thus, it suffices to show that the $K$-coinvariants of
	$$\widetilde{H}_{q}(\bS^{E_1}(\cdot\subseteq V,W\subseteq \cdot|R^n))\otimes \K$$
	vanish. If $V\oplus W=R^n$, then $\bS^{E_1}(\cdot\subseteq V,W\subseteq \cdot|R^n)$ is contractible (as in the proof of \cref{freeCharney}), so there is nothing to prove. On the other hand, if $V\oplus W$ is a proper summand of $R^n$, we can choose a complement $C$ of $W$ that contains $V$ and apply cutting down, exactly as in the proof of \cref{freeCharney}. That is,
	$$\widetilde{H}_{q}(\bS^{E_1}(\cdot\subseteq V,W\subseteq \cdot|R^n))\otimes \K \cong \widetilde{H}_{q}(\bS^{E_1}(\cdot\subseteq V,\cdot|C))\otimes \K.$$
	Notice that every element of $K$ must preserve $C$ (since any $\phi\in K$ determines the identity map on $R^n/V$, it must be the case that $\phi(c)-c\in	V$ for any $c\in C$; hence $\phi(c)\in C$) in addition to preserving $V$
	and fixing $W$, so this isomorphism is $K$-equivariant. Moreover, restriction to $C$ gives an isomorphism from $K$ to $\GL(C,\pres{V},\fix{C/V})$, so it suffices	to show vanishing of the $\GL(C,\pres{V},\fix{C/V})$-coinvariants of
	$\widetilde{H}_{q}(\bS^{E_1}(\cdot\subseteq V,\cdot	|C))\otimes \K$.
	As in the proof of \cref{freeCharney}, \cref{lemma:dualizing-argument-R-vs-Rop} yields a poset isomorphism $\bS^{E_1}(\cdot\subseteq V,\cdot|C) \cong \bS^{E_1}(\cdot,V^{\circ}\subseteq \cdot |C^{\vee})$ that is compatible with the group action of $\GL(C,\pres{V},\fix{C/V})$ on the left side and the group action of $\GL(C^\vee, \fix V^\circ)$ on the right side. It therefore suffices to show vanishing of the
	$\GL(C^{\vee},\fix{V^{\circ}})$-coinvariants of $\widetilde{H}_{q}(\bS^{E_1}(\cdot,	V^{\circ} \subseteq \cdot |C^{\vee}))\otimes \K$. But this is precisely the $\GL(C^{\vee},\fix{V^{\circ}})$-coinvariants of $\on{Ch}(C^{\vee},V^{\circ})\otimes \K$, and the rank of the free $R^{op}$-module $C^{\vee}$ is strictly less than $n$ (since $W$ is nontrivial) and at least 2 (since we have already dealt with the case where $W$ has rank $n-1$), so these coinvariants vanish	by the induction hypothesis for $R^{op}$. This finishes the proof of the second part.
	
	Now we prove the first part, the statement about $(\on{Ch}_n(R)\otimes \K)_{\GL_n(R)}$. The second map-of-posets argument from \autoref{freeCharney} and the second part of \cite[Theorem 4.1]{galatiuskupersrandalwilliams2018cellsandfinite} yield a filtration $0=F_{n-1} \subseteq\dots\subseteq F_{-1}=\St^{E_1}_n(R)\otimes \K$ of $\K[\GL_n(R)]$-modules such that
	$$F_{-1}/F_0\cong \St_n(R)\otimes \K,$$
	and if $q\geq 0$,
	\begin{align*}
		F_q/F_{q+1}&\cong \bigoplus_{\substack{U\in \bT_n(R)\\ \rank{U}=n-q-1}}
		\St(U)\otimes \on{Ch}(R^n,U)\otimes \K\\
		&\cong \Ind_{\GL(R^n,\pres{R^{n-q-1})}}^{\GL_n(R)} \St(R^{n-q-1})\otimes
		\on{Ch}(R^n,R^{n-q-1})\otimes \K,
	\end{align*}
	where $R^{n-q-1}\subseteq R^n$ denotes the submodule spanned by the first $n-q-1$ standard basis vectors. As before, it suffices to show that the coinvariants of each
	quotient $F_q/F_{q+1}$ vanish.
	
	When $q=-1$, the fact that the $\GL_n(R)$-coinvariants of $\St_n(R)\otimes \K$ vanish follows from \cref{corollary:coinvariants-of-steinberg}.	If $q\geq 0$, then Shapiro's lemma says that the $\GL_n(R)$-coinvariants of
	$F_q/F_{q+1}$ are isomorphic to the $\GL(R^n,\pres{R^{n-q-1}})$-coinvariants of
	$$\St(R^{n-q-1})\otimes\on{Ch}(R^n,R^{n-q-1})\otimes \K.$$
	It is therefore sufficient to show that the
	$\GL(R^n,\fix{R^{n-q-1}})$-coinvariants vanish, which is true because
	$\GL(R^n,\fix{R^{n-q-1}})$ acts trivially on $\St(R^{n-q-1})$, and we have
	already shown that $(\on{Ch}(R^n,R^{n-q-1})\otimes \K)_{\GL(R^n,\fix{R^{n-q-1}})}=0$ in the first part of this proof.
\end{proof}
	
	\sloppy
	\printbibliography
\end{document}